\documentclass[a4paper, reqno,12pt]{amsart}
\usepackage{tikz, tikz-3dplot}
\usepackage{amssymb,latexsym, amsmath, amsthm, amscd, array, graphicx, color}
\usepackage{tikz}
\usepackage[english]{babel}
\usetikzlibrary{matrix,arrows,calc}
\usetikzlibrary{positioning}
\usepackage{parskip}
\usepackage[margin=1in]{geometry}
\usepackage{physics}
\usepackage[utf8]{inputenc}
\usepackage{amssymb}
\usepackage[export]{adjustbox}
\usepackage{setspace}
\usepackage{subcaption}
\usepackage{caption}
\usepackage{wrapfig}
\usepackage{bbm}
\usepackage{textcomp}
\usepackage{amsmath}
\usepackage{mathtools, xparse}
\usepackage{enumitem}
\usepackage{physics}
\usepackage{textcomp}
\usepackage{siunitx}
\usepackage{mathtools}
\usepackage{graphics}
\usepackage{mathrsfs}
\usepackage{url}
\theoremstyle{plain}
\newtheorem{theorem}{Theorem}
\newtheorem{lemma}[theorem]{Lemma}
\newtheorem{proposition}[theorem]{Proposition}
\newtheorem{corollary}[theorem]{Corollary}

\newtheorem{conjecture}[theorem]{Conjecture}
\theoremstyle{definition}
\newtheorem{remark}[theorem]{Remark}

\def\C{{\mathbb C}}

\def\R{{\mathbb R}}
\def\N{{\mathbb N}}

\DeclarePairedDelimiterX{\inp}[2]{\langle}{\rangle}{#1, #2}
\newcommand{\unorm}[1]{{\left\vert\kern-0.25ex\left\vert\kern-0.25ex\left\vert #1 
    \right\vert\kern-0.25ex\right\vert\kern-0.25ex\right\vert}}

\title{Random Constructions for Sharp Estimates of Mizohata-Takeuchi Type}
\author{Siddharth Mulherkar}
	\address{Department of Mathematics, University of California, Los Angeles, USA}
	\email{sidmulherkar@g.ucla.edu}
\date{\today}
\begin{document}
\maketitle
\begin{abstract}
Let $\Sigma \subseteq \mathbb{R}^d$ be a smooth, compact hypersurface with natural surface measure $d \sigma$, and let $Eg$ be its associated Fourier extension operator. A \textit{Mizohata-Takeuchi} type estimate is an estimate of the form $\int_{B_R}\abs{Eg(x)}^2 w(x) dx \leq C_{\epsilon}R^{\epsilon} \norm{Xw}_{\infty}\int_{\Sigma} \abs{g(\omega)}^2 d\sigma$, where $w: \mathbb{R}^d \rightarrow \R^{+}$ is an arbitrary positive \textit{weight} function and $Xw$ is its X-ray transform.  Using tools from high-dimensional probability, we construct a large collection of weights $w$ that satisfy sharp inequalities of the Mizohata-Takeuchi type. These weights are constructed randomly, and one can interpret our result as saying that with high probability, a generic weight satisfies a sharp inequality of Mizohata-Takeuchi type. 
\end{abstract}

\maketitle
\section{Introduction}

Let $\Sigma \subseteq \R^d$ be a compact, smooth hypersurface with surface measure $d \sigma$. Given a measurable map $g: \Sigma \rightarrow \C$, we define the extension operator $Eg: \R^{d} \rightarrow \C$, by $Eg(x) = \int_{\Sigma} e^{2 \pi i \omega \cdot x} g(\omega) d \sigma (\omega)$. This paper is concerned with the study of weighted $L^{p}$ estimates for the operator $Eg$. Whenever $\Sigma$ has strictly positive second fundamental form, Stein's \textit{restriction conjecture} asserts that $\norm{Eg}_{L^{p}(\mathbb{R}^d)} \leq C_{p,d} \norm{g}_{L^{\infty}(\Sigma)}$ for $p>\frac{2n}{n-1}$ and for all smooth functions $g: \Sigma \rightarrow \mathbb{C}$. Stated differently, the restriction conjecture asks for sharp bounds for the quantity 
\begin{equation}\label{restrictionlevelset}
    \int_{0}^{\infty} \lambda^{p-1} \abs{\{x \in \R^d: \abs{Eg(x)}> \lambda\}} d\lambda.
\end{equation}
given the normalization $\norm{g}_{L^{\infty}(\Sigma)}=1$. One can interpret (\ref{restrictionlevelset}) as asking for sharp bounds for the volume of the superlevel sets $\{x: \abs{Eg(x)} > \lambda \}$. The \textit{Mizohata-Takeuchi Conjecture} is a conjecture that asks for more refined, qualititative information about the superlevel sets by probing into the \textit{shape} of the super level sets $\{x: \abs{Eg(x)}> \lambda \}$. Motivated by questions surrounding the well-posedness of first order perturbations of certain dispersive PDE, Mizohata \cite{mizohata1985cauchy} and Takeuchi \cite{takeuchi1974necessary}, \cite{takeuchi1980cauchy} conjectured\footnote{The exact history surrounding this conjecture is actually quite ambiguous: see \cite{carbery2022mizohata}, \cite{carbery2024sharp}, \cite{cairo2025counterexample} for some more detailed historical overview.} the following: 
\begin{conjecture}[Mizohata--Takeuchi]\label{mtconj}

    Let $w: \R^d \rightarrow \R^{+}$ be a positive \textit{weight} function and $\Sigma \subseteq \R^d $ is a compact, convex, and smooth hypersurface, then
    \begin{equation}\label{originalmt}
        \int_{\R^d}\abs{Eg(x)}^2w(x)dx \lesssim \norm{Xw}_{\infty} \int_{\Sigma}\abs{g(\omega)}^2 d \sigma(\omega).
    \end{equation}
    \end{conjecture}
    Here $Xw$ is the $X-$ray transform of $w$ so that
    \begin{equation*}
        \norm{Xw}_{\infty} := \sup_{\ell} \int_{\ell} w(x)dx, 
    \end{equation*}
    where the $\sup$ is taken over all possible lines $\ell \subseteq \R^d$ in all possible directions. By considerations coming from the uncertainty principle (since $\Sigma$ is compact), it is known that (\ref{originalmt}) is equivalent to
    \begin{equation}\label{nonoriginalmt}
        \int_{\R^d}\abs{Eg(x)}^2w(x)dx \lesssim \sup_{T \in \mathbb{T}} w(T) \int_{\Sigma}\abs{g(\omega)}^2 d \sigma(\omega) ,
    \end{equation}
    where $\mathbb{T}$ is the collection of all 1-tubes in $\mathbb{R}^d$, here a $1$- tube is a 1- neighborhood of a line, and $w(T)$ is defined as $\int_{T}w(x)dx$. Note that there is no assumption on the curvature here: (\ref{nonoriginalmt}) is known to be true when $\Sigma$ is a hyperplane. Considerations coming from the uncertainty principle also show that there is no loss of generality in assuming that $w$ is constant at every unit scale. 
    
    A prototypical example of a weight is a the characteristic function of a union of unit balls, that is when $w(x) \coloneqq \sum_k \mathbbm{1}_{\alpha_k}(x)$ for $\{\alpha_k\}_k$ being an arbitrary collection of unit balls in $\R^d$. Coming up with sharp bounds for the quantity $\int_{B_R}\abs{Eg(x)}^2w(x)dx$ for different weights $w$ will give us an insight into the shape of the super level sets $\{x: \abs{Eg(x)} > \lambda \}$. By contrast, the restriction conjecture only asks for sharp bounds for the \textit{volume} of the set $\{x: Eg(x)> \lambda \}$. 
    
    The Mizohata-Takeuchi conjecture is a problem that falls into the broader category of \textit{weighted restriction estimates}. The study of weighted restriction estimates has been a crucial input that has led to some exciting progress in a variety of problems, such as the resolution of Carleson's problem for the pointwise convergence of the Schr\"odinger equation, up to the endpoint (see \cite{du2019sharp}) and the Falconer distance problem (see, for example \cite{guth2020falconer} or \cite{du2023falconer}). 
    
    While there has been an influx of work that study weighted restriction estimates in various settings with many exciting and interesting results (see for example, \cite{ortiz2023sharp}, \cite{cairo2025counterexample}, \cite{carbery2024sharp}, \cite{du2019sharp}), to the best of the author's knowledge, a study of weighted restriction estimates in the case of a generic, random class of weights has not been undertaken yet. This work aims to fill this gap in the literature.

    Aside from probing into the shape the superlevel sets, the Mizohata-Takeuchi conjecture is also of general significance in harmonic analysis. The Mizohata-Takeuchi conjecture is implied by a more general conjecture of Stein, which in turn is related to the connection between the Kakeya and Nikodym maximal functions with Bochner--Riesz multipliers (see \cite{bennett2006stein}). It is also related to the multilinear restriction conjecture. The Mizohota-Takeuchi conjecture, if true for sphere, would in fact prove the (endpoint) multilinear restriction conjecture \cite{carbery2023disentanglement}. The interested reader is referred to Carbery's talk \cite{carbery2022mizohata} for a general overview of the conjecture, its significance, and partial results in the positive direction.

    There have been some recent developments surrounding the Mizohata-Takeuchi conjecture that are relevant to this work which we will summarize here before we explain the main contribution of this paper. 

     One way to measure progress towards Conjecture \ref{mtconj} is to localize to $B_R$, the ball of radius $R$ around the origin in $\R^d$, and to determine the best possible exponent $\alpha \geq 0$ for which the following inequality is true:
    \begin{equation}\label{localizedmt}
        \int_{B_R}\abs{Eg(x)}^2w(x)dx \lesssim R^{\alpha}\sup_{T \in \mathbb{T}} w(T) \int_{\Sigma} \abs{g(\omega)}^2 d\sigma(\omega).
    \end{equation}
    (Conjecture \ref{mtconj} asserts that the above inequality should hold for $\alpha =0$).

Using ideas related to decoupling (more specifically, the refined decoupling theorem, which first appeared in \cite{guth2020falconer}), Carbery, Illiopoulou, and Wang in \cite{carbery2024sharp} showed that (\ref{localizedmt}) holds for $\alpha = \frac{d-1}{d+1}$ (up to an $\epsilon$-loss), under a natural assumption about curvature and smoothness of $\Sigma$. More specifically, the following was shown:
\begin{theorem}[\cite{carbery2022mizohata}]
    Let $\Sigma$ be a compact $C^2$ hypersurface with nonvanishing Gaussian curvature. Then
    \begin{equation}\label{ciw}
        \int_{B_R}\abs{Eg(x)}^2dxw(x) \lesssim_{\epsilon} R^{\frac{d+1}{d-1}+\epsilon}\sup_{T \in \mathbb{T} } w(T) \int_{\Sigma} \abs{g(\omega)}^2 d\sigma(\omega).
    \end{equation}
\end{theorem}

Another very important development in this direction is that Cairo \cite{cairo2025counterexample} has shown that 
Conjecture \ref{mtconj} is false (up to a logarithmic factor in $R$)! 
\begin{theorem}[\cite{cairo2025counterexample}] \label{counterexamplecairo}
    For any $C^2$ hypersurface $\Sigma$ that is not a plane, there exists a weight $w: \R^d \rightarrow \R^{+}$ and $g: \Sigma \rightarrow \C$ such that
    \begin{equation*}
        \int_{B_R}\abs{Eg(x)}^2w(x) dx \gtrsim \log R \sup_{ T \in \mathbb{T}} w(T)\int_{\Sigma} \abs{g(\omega)}^2 d\sigma(\omega).
    \end{equation*}
    \end{theorem}
    In light of Theorem \ref{counterexamplecairo} the following ``$\epsilon-$loss'' version of Conjecture \ref{mtconj} was posed in \cite{cairo2025counterexample}.
    \begin{conjecture}[Local Mizohata--Takeuchi]\label{localmtconj}
    
        Under the hypotheses of Conjecture \ref{mtconj}, the following holds:
        \begin{equation}\label{localmtconjeq}
            \int_{B_R}\abs{Eg(x)}^2w(x)dx \lesssim_{\epsilon} R^{\epsilon}\sup_{T \in \mathbb{T}} w(T) \int_{\Sigma} \abs{g(\omega)}^2 d\sigma(\omega).
        \end{equation}
        \end{conjecture}
        Conjecture \ref{localmtconj} remains open; however, it is not clear in this setting whether one should expect the optimal exponent of $R$ in (\ref{localizedmt}) to be $\frac{d-1}{d+1}$, coming from (\ref{ciw}), or if Conjecture \ref{localmtconj} is true. See \cite{guth2022enemy} and \cite{carbery2024sharp} for a related discussion about the sharpness of (\ref{ciw}). We now turn to our main result, which is a construction of random weights that satisfy an inequality of the type (\ref{localmtconjeq}).
        \subsection{Random Construction of Weights and Main Result}

        A natural question that arises in the study of the Mizohata Takuechi conjecture is what kind of weights should be test this conjecture on? We analyse the case where $w$ is the characteristic function of a union of unit balls contained in $B_R$. The classical Tomas-Stein restriction theorem states that
        \begin{equation}
            \norm{Eg}_{{\frac{2(d+1)}{d-1}}} \lesssim \norm{g}_{L^2(\Sigma)}.
        \end{equation}
        Combining this with H\"older, we obtain
        \begin{equation*}
            \int\abs{Eg(x)}^2 w(x)dx \lesssim \norm{w}_{\frac{d+1}{2}} \int_{\Sigma} \abs{g(\omega)}^2 d \sigma(\omega)
        \end{equation*}
        Comparing this with (\ref{localmtconjeq}), notice that if $\underset{\text{1-tubes } T}{\sup} w(T) \sim \norm{w}_{\frac{d+1}{2}}$, then the conjecture is readily true. Therefore, in order to make progress on Conjecture \ref{localmtconj} one needs to understand the case when $\underset{\text{1 - Tubes } T}{\sup} w(T)$ is much smaller than $\norm{w}_{\frac{d+1}{2}}$, which is the case when our weight has ``large mass'' and ``low tube occupancy''. In \cite{carbery2009large}, Carbery constructed an interesting class of weights, which had the two aformentioned properties: 
        \begin{theorem}[\cite{carbery2009large}]
            There exist weights $w: B_{R} \rightarrow \R^{+}$ such that $w$ is \textit{essentially}\footnote{In Carbery's construction it is not quite true that $w$ is the characteristic function of a union of unit balls, rather the weight is constructed using a random replacement model, which is comparable to the case when $w$ the characteristic function of a union of unit balls, more of this is discussed in Proposition \ref{carberycouplingprop}.} a characteristic function of a union of unit balls such that 
            \begin{equation}\label{carberyweightconstructionineq}
                \sup_{T \in \mathbb{T}} w(T) \lesssim_{\epsilon} R^{\epsilon} \quad \text{and} \quad \norm{w}_{1} \sim_{\epsilon} R^{d-1}.
            \end{equation}
            \end{theorem}
            Notice that if one requires $\underset{{\text{1-tubes } T}}{\sup} w(T) \lesssim R^{\epsilon}$, then the largest $\norm{w}_1$ can be (up to subpolynomial factors in $R$) is $\sim R^{d-1}$. Thus one can interpret this result as the construction of weights that maximizes (in a suitable sense) the ratio of ``mass'' and ``tube occupancy''. The weight $w$ is constructed randomly\footnote{See \cite{carbery2009large} and section \ref{carberymodelrandomness} in this paper for more details and discussion.}\footnote{To the best of the author's  knowledge, no explicit examples of weights satisfying the inequalities in (\ref{carberyweightconstructionineq}) are known.}, and with respect the model of randomness used to generate these weights, Carbery showed the estimates in (\ref{carberyweightconstructionineq}) hold with high probability ($\geq 1/2)$. This means that there is a large class of weights $w$ with $\text{supp}(w) \subseteq B_R$ for which the estimates in (\ref{carberyweightconstructionineq}) are true. We address the question of whether many of these weights satisfy sharp inequalities of the Mizohata-Takeuchi type in this paper. We are now ready to state (informally) the main contribution of this paper.
            \begin{theorem}[Main Theorem, Informal Statement]\label{maintheoreminformal} 
            
                 We randomly construct weights $w: \R^d \rightarrow \R$ with $\text{supp}(w) \subseteq B_R$ such that with high probability (say, $\geq \frac{1}{2}$) the following hold:
                \begin{enumerate}
                    \item $\int_{B_R}\abs{Eg(x)}^2 w(x) dx \lesssim_{\epsilon,d} R^{\epsilon}\int_{\Sigma}\abs{g(\omega)}^2 d\sigma$, uniformly for all $g \in L^2(\Sigma)$.
            \item $\sup_{T \in \mathbb{T}} w(T) \lesssim_{\epsilon} R^{\epsilon}$.
            \item $\norm{w}_{1} \sim R^{d-1}$.
                \end{enumerate}
                \end{theorem}
                This can be interpreted as saying that a generic class of weights satisfy sharp inequalities of the Mizohata-Takeuchi type (up to an $\epsilon-$loss).
            \begin{remark}
                With little more effort, by modifying the random construction in Theorem \ref{maintheoreminformal}, we can show that a slight generalization of Theorem \ref{maintheoreminformal} can be obtained: \end{remark}
            \begin{theorem}[Mild Generalization of Theorem \ref{maintheoreminformal}]\label{mildgeneralization}
            
                 We can randomly construct weights $w: \R^d \rightarrow \R^{+}$ with $\text{supp}(w) \subseteq B_R$ such that with high probability the following hold for each $0 \leq \lambda < 1 $:
                \begin{enumerate}
                    \item $\int_{B_R}\abs{Eg(x)}^2 w(x) dx \lesssim_{\epsilon,d} R^{\epsilon + \lambda}\int_{\Sigma}\abs{g(\omega)}^2 d\sigma$, uniformly for all $g \in L^2(\Sigma)$.
            \item $\sup_{T \in \mathbb{T}} w(T) \lesssim_{\epsilon} R^{\epsilon+ \lambda}$.
            \item $\norm{w}_{1} \sim R^{d-1+ \lambda}$.
                \end{enumerate}
            \end{theorem}
            Since the proof of Theorem \ref{mildgeneralization} is almost exactly the same as that of Theorem \ref{maintheoreminformal}, we solely focus on proving Theorem \ref{maintheoreminformal}.
        \subsection{Preparation for the Main Theorem and Proof Overview}

        \subsection{Construction of the Random Weight $w$:}

        In order to prove Theorem \ref{maintheoreminformal}, we begin by setting up a random weight model. Decompose $B_R$ into a collection of finitely overlapping unit balls $\{\alpha_k\}$ that cover $B_R$. Since $\text{Vol}(B_R) \sim R^{d}$, we may index the balls as $\{\alpha_k\}_{k=1}^{CR^{d}}$ for some absolute constant $C$. 
        
        Our construction of a random weight is quite elementary and informally its description is as follows: at each site $\alpha_k$ we independently flip a coin with probability of success $O(\frac{1}{R})$, if we are successful at site $\alpha_k$ then we add the characteristic function of $\alpha_k$ to our weight $w$.

        More precisely, let $\{\delta_k\}_{k=1}^{CR^d}$ be an i.i.d. collection of selector random variables where $\mathbb{P}(\delta_k = 1) := \delta = O(\frac{1}{R})$ and $\mathbb{P}(\delta_k = 0) = 1- \delta$, for each $1 \leq k \leq CR^d$. We set $w(x) := \sum_{k=1}^{CR^d} \delta_k \mathbbm{1}_{\alpha_k} (x)$. From here we see that $\mathbb{E} (\norm{w}_1) \sim \mathbb{E}(\sum_{k=1}^{CR^d}\delta_k) \sim R^{d-1}$, since each of the $\alpha_k$'s have measure $\sim 1$ and are finitely overlapping. 
        
        From the previous line, it is clear that with high probability the random weight $w$ will be such that $\norm{w}_1 \sim R^{d-1}$, so property 3 of Theorem \ref{maintheoreminformal} is satisfied. The fact that property $w$ satisfies property 2 of Theorem \ref{maintheoreminformal} is a consequence of a technical (but not too difficult to prove) large deviation bound for the sum of Bernoulli random variables which will be proved in Lemma \ref{bernoullilargedevbound} later in this paper.
        
        The main technical difficulty lies in establishing that property 1 holds with high probability. We show property 1 is true by establishing the following theorem:
        \begin{theorem}\label{maintechnicaltheorem}
            In the setting coming from the discussion above, the following holds:
            \begin{equation}
            \mathbb{E}\sup_{\norm{g}_{L^2(\Sigma)} \leq 1} \int\abs{Eg(x)}^2w(x)dx \lesssim R^{\epsilon}.
            \end{equation}
        \end{theorem}
            Theorem \ref{maintechnicaltheorem} shows that with high probability the weight $w$ satisfies the inequality
            \begin{equation*}
            \sup_{\norm{g}_{L^2(\Sigma)} \leq 1} \int\abs{Eg(x)}^2w(x)dx \lesssim R^{\epsilon},
            \end{equation*}
            which in turn implies (by rescaling) that 
                $\int_{B_R}\abs{Eg(x)}^2 w(x) dx \lesssim_{\epsilon,d} R^{\epsilon}\int_{\Sigma}\abs{g(\omega)}^2$, which proves property 3.  By unwinding the definition of $w$, note that Theorem \ref{maintechnicaltheorem} is equivalent to the following proposition (which we will refer to the most along the way to proving Theorem \ref{maintechnicaltheorem}):
                \begin{proposition}\label{technicalrefinedprop}
                    Let $\{\alpha_k\}_{k=1}^{CR^d}$ be the finitely overlapping collection of unit balls covering $B_R$ as defined above, and let $\{\delta_k\}_{k=1}^{CR^d}$ be the corresponding i.i.d. collection of selector random variables with $\mathbb{P}(\delta_k =1):= \delta =O(\frac{1}{R})$, then we have
                    \begin{equation}\label{maintechnicalpropeq}
                        \mathbb{E}\sup_{\norm{g}_{L^2(\Sigma) \leq 1}} \sum_{k=1}^{CR^d}\delta_k \int_{\alpha_k}\abs{Eg(x)}^2dx \lesssim_{\epsilon, d} R^{\epsilon}.
                    \end{equation}
                \end{proposition}

\subsection{Suprema of Stochastic Processes:}

The key inspiration behind the proof of Proposition \ref{technicalrefinedprop} is to use ideas from the theory of bounding the expected supremum of stochastic processes. We can reframe (\ref{maintechnicalpropeq}) more abstractly as a question that asks for tight bounds for the quantity
\begin{equation}\label{expectedsup}
    \mathbb{E}\sup_{t \in T} X_t,
\end{equation}
for some stochastic process $(X_t)_{t \in T}$. There is an incredibly rich and interesting theory surrounding this family of problems and we refer the reader to Talagrand's book \cite{talagrand2014upper} and the references therein for a comprehensive introduction. The study of the expected suprema of Gaussian processes is especially comprehensive; when $(X_t)_{t \in T}$ is a Gaussian process it is known that the quantity (\ref{expectedsup}) is completely determined by the metric induced by the $L^2$ distance between random variables on the index set $T$, given by $d(s,t) = (\mathbb{E}(X_t-X_s)^2)^{1/2}$ for $s,t \in T$. A powerful strategy to used bound the expected supremum of a Gaussian process is called the \textit{chaining} method. An informal description of chaining for Gaussian processes is as follows: assuming the process is centered\footnote{That is, $\mathbb{E}X_{t} =0$ for all $t \in T$.} we can write
\begin{equation*}
    \mathbb{E}\sup_{t \in T}X_t = \mathbb{E}\sup_{t \in T} (X_t-X_{t_0}),
\end{equation*}
and then write $(X_t-X_{t_0}) = \sum_{k} (X_{\pi_{k}(t)}-X_{\pi_{k-1}(t)}) $ as some telescoping sum, where the $\pi_{k}(t)$'s can be thought of as successive approximations to $t$ in the metric space $(T,d)$. One can then use measure concentration bounds\footnote{In the case of a Gaussian process, one usually uses the Gaussian tail bound: $\mathbb{P}(\abs{X_s-X_t}> \lambda) \leq 2 \exp\left(\frac{-\lambda^2}{2d(s,t)^2}\right)$.} to control the size of the increments $X_{\pi_{k}(t)}- X_{\pi_{k-1}(t)}.$  A celebrated application of the chaining method is \textit{Dudley's inequality} which states in this setting that
\begin{equation*}
    \mathbb{E} \sup_{t \in T}X_t \lesssim \int_{0}^{\infty} \sqrt{\log N(T,d, \epsilon)} d \epsilon,
\end{equation*}
where $N(T,d, \epsilon)$ is the $\epsilon$ - \textit{covering number} of $T$, that is, the minimum number of balls of radius $\epsilon$ in the metric space $(T,d)$ required to cover $T$. We refer the reader to the references \cite{talagrand2014upper} and \cite{vershynin2018high} for excellent introductions to the chaining method and applications. Our method for proving Proposition \ref{technicalrefinedprop} comes from the chaining method, along with some basic tools from restriction theory and is inspired by the important works \cite{bourgain1989bounded}, \cite{Talagrand1992}, and \cite{talagrand1994majorizing}. We summarize some of the ideas in these works and outline how they will be used to prove Proposition \ref{technicalrefinedprop}.

\begin{remark}
    The astute reader may have realized that the quantity
    \begin{equation*}
        \mathbb{E} \sup_{t \in T}X_t
    \end{equation*}
    may not be well defined if $T$ is uncountable (since the random variable $\sup_{t \in T}X_t$ may not even be measurable). The fix for this is to alter the definition of the expected supremum of a stochastic process so that all the previous arguments make sense by defining
    \begin{equation*}
        \mathbb{E}\sup_{t \in T}X_t \coloneqq \sup_{F \subseteq T, \text{ finite}} \mathbb{E}\sup_{t \in F}X_t.
    \end{equation*}
    We will proceed with this convention for this paper as well.
\end{remark}

\subsection{Random Fourier Series and Ideas Behind the proof}

In the celebrated paper \cite{bourgain1989bounded} Bourgain managed to prove sharp $L^p$ bounds for a certain type of random Fourier series. In particular, Bourgain showed that a generic random subset $J$ of the discrete interval $\{1,2,...,n\}$ of size $n^{2/p}$ satisfies the inequality 
\begin{equation}\label{lambdap}
    \norm{\sum_{k\in J} a_k e^{ikx}}_{L^{p}[0,2\pi]} \lesssim \left(\sum_{k \in J}\abs{a_k}^2\right)^{\frac{1}{2}},
\end{equation}
where the implicit constant in the above inequality is an absolute and does not depend on $n$. Subsequently, Talagrand \cite{Talagrand1992} offered another proof of this theorem. While both proofs significantly differ in methodology, they both invoke ideas coming from the chaining method and the (metric) geometry of stochastic processes. Bourgain uses a variant of Dudley's inequality, whereas Talagrand uses refined ideas coming from his \textit{majorizing measures theorem} - which yields precise bounds for the expected suprema of Gaussian processes\footnote{Although modifications to arguments coming from the majorizing measure theorem are needed, since the relevant stochastic process in this case is not a Gaussian process. Many of ideas to generalize the scheme coming from the majorizing measure theorem to Bernoulli processes was first developed in \cite{talagrand1994majorizing}.}. 

We follow the method developed by Talagrand in \cite{Talagrand1992} (with appropriate modifications). In particular, Talagrand observes that in order to show (\ref{lambdap}) it is equivalent to showing that
\begin{equation}\label{tallambdap}
    \mathbb{E}\sup_{\norm{f}_{{p'}} \leq 1}\sum_{k=1}^{n} (\tilde{\delta_k}-\tilde{\delta}) \abs{\hat{f}(k)}^2 \lesssim 1,
\end{equation}
where $\{\tilde{\delta_k}\}_{k=1}^{n}$ is an i.i.d. collection of selector random variables with $\mathbb{P}(\tilde{\delta_k}=1):=\tilde{\delta}= n^{2/p-1}$. We can compare this to what Proposition \ref{technicalrefinedprop} which we recall is the statement that
\begin{equation}\label{techref}
    \mathbb{E}\sup_{\norm{g}_{L^2(\Sigma) \leq 1}} \sum_{k=1}^{CR^d}\delta_k \int_{\alpha_k}\abs{Eg(x)}^2dx \lesssim_{\epsilon, d} R^{\epsilon}.
\end{equation}
Both (\ref{tallambdap}) and (\ref{techref}) have a similar structure, and we adopt the method of \cite{Talagrand1992}. Some improvements and modifications are needed, which we will describe along the way.

In order to prove Proposition \ref{technicalrefinedprop}, in the spirit of the chaining method, we need to understand the geometry of the metric space $\mathcal{B} := \{\ g \in L^2(\Sigma): \norm{g}_{L^2(\Sigma)} \leq 1 \}$ with respect to an appropriate metric $d(\cdot,\cdot)$. This is done by gaining control of the covering numbers $N(\mathcal{B}, d, \epsilon)$ using the empirical method of Maurey (see \cite[Chapter 0]{vershynin2018high}) and a duality result\footnote{We are able to make use of this result to make our work easier compared to what was done in \cite{Talagrand1992}.} from \cite{Artstein2004}. Following this, we use a chaining-type method which closely resembles the method in \cite{Talagrand1992} by relating measure concentration bounds to covering numbers (similar to how one does in the proof of Dudley's inequality). These steps only allow us to prove a slightly weaker result than Proposition \ref{technicalrefinedprop}. We first show the following weaker version of Proposition \ref{technicalrefinedprop}:
\begin{equation*}
    \mathbb{E}\sup_{\norm{g}_{L^2(\Sigma) \leq 1}} \sum_{k=1}^{CR^d}\delta_k \abs{\int_{\alpha_k}Eg(x)dx}^2 \lesssim_{\epsilon, d} R^{\epsilon},
\end{equation*}
and in order to upgrade this, we use some tools from restriction theory (namely, the locally constant property for $Eg$ to rerun the entire argument) to complete the proof of Proposition \ref{technicalrefinedprop}.

\subsection{Structure and Organization of the Paper:} 
\begin{itemize}
    \item In section \ref{sectionweakproof}, we prove a statement that is weaker than Proposition \ref{technicalrefinedprop}. Many of the steps here follow the scheme in \cite{Talagrand1992}, with some appropriate modifications.
    \item In section \ref{sectionprop9} we use local constancy estimates and a modification of the proof in the previous section to prove Proposition \ref{technicalrefinedprop}.
    \item In section \ref{probestimatesection} we establish some probabilistic estimates which will complete the proof of Theorem \ref{maintheoreminformal}. We also do a comparison with a random weight model proposed by Carbery in \cite{carbery2009large}, and show that our results apply to his model as well.
    \item The appendix contains proofs of some lemmas in probability and Fourier analysis. Some of these lemmas are standard and in some cases well known, but we include them for the sake of completeness and for the reader that is not so familiar with one of these fields. We have tried to make the exposition of this article as accessible as possible. 
\end{itemize}

\textbf{Notation:} We will use the following notation in this paper:
\begin{itemize}
    \item $A \lesssim B$ if there exists an absolute constant $C>0$ such that $A \leq C B$.
    \item $A \sim B$ if there exists an absolute constant such that $\frac{1}{C}B \leq A \leq CB$.
    \item $A \lesssim_{\epsilon} B$ if for every $\epsilon >0 $ there exists a constant $C_{\epsilon}>0$ (possibly depending on $\epsilon$) such that $A \leq C_{\epsilon} B$ for all $\epsilon >0$. For example, $\log R \lesssim_{\epsilon} R^{\epsilon}$ for all $R>0$.
    \item Given any any measure space $(X, \mu)$, and for $1 < p< \infty$ we define the $L^p(X)$ to be its usual definition (i.e. complex valued measurable functions with finite $p^{th}$ moment). We similarly define $L^{p}(X, \mathbb{R})$ to be the the set of \textit{real} valued measurable functions with finite ${p}^{th}$ moment. We often need to deal with functions in $L^{p}(X, \mathbb{R})$ for some suitable measure space $X$. For $g \in L^{p}(X, \mathbb{R})$, we define $\norm{g}_{L^p(X, \mathbb{R})} = (\int \abs{g}^{p} d\mu )^{1/p}$.
    \item Let $w: \mathbb{R^d} \rightarrow [0, \infty)$, be a positive \textit{weight} function. We define $\norm{f}_{L^p(w)} = (\int_{\mathbb{R}^d}{\abs{f(x)}^pw(x)}dx)^{1/p}$.
    \item Given any $A, B \subseteq \mathbb{R}^n$, define $\mathcal{N}(A, B)$ to be the minimum number of translates of $B$ needed to cover $A$. If $\norm{\cdot}$ is a norm on $\mathbb{R}^n$, we write $\mathcal{N}(A, \norm{\cdot}, \epsilon)$ to be the minimum number of ball of radius $\epsilon$ in $\norm{\cdot}$ needed to cover $A$.  Covering numbers of the form $\mathcal{N}(A,B)$ are a basic tool in the theory of Gaussian processes and will be used frequently in this work.
    \item $\mathbb{T}$ will denote the collection of all 1-tubes (1- neighborhoods of lines) in $\R^d$
    \item When the context is clear, we let $\Sigma$ be be a compact hypersurface with $\Sigma \subseteq \R^d$. By scaling and translating, we may assume that $\Sigma \subseteq \{x \in \R^d: \abs{x} \leq 1 \}$.
    \item $\mathcal{B}$ will always denote the unit ball in $L^2(\Sigma)$, that is $\mathcal{B} := \{g: \norm{g}_{L^2(\Sigma)} \leq 1 \}$, however, depending on the context, the functions in $\mathcal{B}$ may take real or complex values. In section 2, we will denote $\mathcal{B} \coloneqq \{g: \Sigma \rightarrow \R: \norm{g}_{L^2(\Sigma)} \leq 1\}$, whereas in Section 3, $\mathcal{B} \coloneqq \{g: \Sigma \rightarrow \C: \norm{g}_{L^2(\Sigma)} \leq 1\}$. This slight abuse of notation is done for the ease of presentation. 
\end{itemize}
\textbf{Acknowledgments:} The author would like to thank Hong Wang and Terence Tao for their encouragement, patience, and invaluable guidance throughout the course of this project. We would also like to thank Arian Nadjimzadah for many fruitful discussions on restriction theory, and for his continued encouragement and interest in this work. 
\section{Proof of Weaker Version of Proposition \ref{technicalrefinedprop}}\label{sectionweakproof}
Let $(\alpha_k)_{k=1}^{CR^d}$ be a finitely overlapping cover of unit balls that cover $B_R$ where $C$ is an absolute constant. We shall show the following result.
\begin{proposition}\label{mainprop} Let $\delta =O(\frac{1}{R})$. For each $1 \leq k \leq n$, Let $(\delta_k)_{k=1}^{n}$ be i.i.d. Bernoulli random variables with $\mathbb{P}(\delta_k=1) = \delta$, then
   \begin{equation}\label{cos}
       \mathbb{E} \sup_{\norm{g}_{{L^2}(\Sigma)} \leq 1} \sum_{k=1}^{R^d} \delta_k \abs{\int_{\alpha_k} Eg(x) dx }^2 \lesssim (\log R)^{O(1)}
   \end{equation}
\end{proposition}
For the ease of presentation we prove a slightly weaker result where our functions take values over $\R$ and not $\mathbb{C}$. We show that
\begin{equation}\label{complex}
    \mathbb{E} \sup_{\norm{g}_{L^2(\Sigma, \R) \leq 1}} \sum_{k=1}^{CR^d} \delta_k \abs{\int_{\alpha_k} \Tilde{E}g(x) dx }^2 \lesssim (\log R)^{O(1)}
\end{equation}
where $\Tilde{E}g(x) := \int_{\Sigma}g(\omega)\cos(2 \pi \omega \cdot x) d \sigma(\omega)$. 

By repeating the same argument with $\sin (2 \pi \omega \cdot x)$ instead of $\cos (2 \pi \omega \cdot x)$ will show that (\ref{cos}) implies (\ref{complex}). We will revisit this claim at the end of this section.

A key idea is to gain control of covering numbers of polytopes in Hilbert space. This goes back to work of Maurey, which is often referred to as the \textit{empirical method}. Maurey's method can be adapted to more general settings, and for our purpose we require the following version.
\begin{theorem}\label{maurey}
    Let $(\mathcal{H}, \norm{\cdot})$ be a Hilbert space (over $\R$). Consider $y_1, y_2,...y_n \in \mathcal{H}$ and suppose $\norm{y_k} \leq K$ for some absolute constant $K$. Let $\mathcal{C} = \{\sum_{k=1}^{n} \alpha_k y_k : \sum_{k=1}^{n} \abs{\alpha_k} \leq 1 \}$(that is, $\mathcal{C}$ is the balanced convex hull of $\{y_1,..., y_n\}$). Then $\log \mathcal{N}(\mathcal{C}, \norm{\cdot}, \epsilon) \lesssim \frac{\log n}{\epsilon^2}$.  
\end{theorem}
\begin{proof}
    Fix $x \in \mathcal{C}$. We may write $x = \sum_{k=1}^{n} \lambda_k y_k$ with $\sum_{k=1}^{n}{\abs{\lambda_k}} \leq 1$. Define the ($\mathcal{H}$ valued) random variable $Y$ with the following distribution: 
    \begin{align*}
        \mathbb{P}(Y=y_k) &= \lambda_k \quad \text{for 
  } 1 \leq k \leq n \\
   \mathbb{P}(Y=0) &= 1 - \sum_{k=1}^{n} \abs{\lambda_k}
    \end{align*}
    Clearly, $\mathbb{E}(Y) = x$. Let $Y_1,..., Y_k$ be i.i.d. copies of $Y$. We now have the following estimate:
    \begin{align*}
        \mathbb{E}\norm{\frac{1}{k}\left(\sum_{j=1}^{k} Y_j \right) - x}^2 &= \frac{1}{k^2}\mathbb{E} \norm{\sum_{j=1}^{k} (Y_j - x)}^2\\
        &=\sum_{1 \leq i, j \leq n} \mathbb{E} \inp{(Y_i-x)}{(Y_j-x)}\\
        &= \frac{1}{k^2}\sum_{j=1}^{k}\mathbb{E}\norm{Y_j-x}^2\\
        &\lesssim \frac{1}{k} \quad \text{(since $\norm{Y_j- x}\leq 2K$.)}
    \end{align*}
    where the second line follows from the following fact that if $Z_1$ and $Z_2$ are independent $\mathcal{H}$-valued random variables with mean $0$, then $\mathbb{E}\inp{Z_1}{Z_2} = 0$. We refer the reader to Lemma \ref{indepexpectation} of the appendix.

    As a consequence of this estimate we note that there is a \textit{realization} of $(Y_1,..., Y_k)$ such that 
    \begin{equation*}
        \norm{\frac{1}{k}\left(\sum_{j=1}^{k} Y_j \right) - x} \lesssim \frac{1}{\sqrt{k}},
    \end{equation*}
    and thus for $k$ with $k \sim \lceil{\frac{1}{\epsilon^2}}\rceil$, then there is a realization of $(Y_1,...,Y_k)$ such that 
    \begin{equation}\label{epsilon_cover}
        \norm{\frac{1}{k}\left(\sum_{j=1}^{k} Y_j \right) - x} \leq \epsilon.
    \end{equation}
    For such $k$, let $\mathcal{F} \subset \mathcal{H}$ be the set 
    \begin{equation*}
        \left\{\frac{1}{k} \sum_{i=1}^{k} x_i : x_1,...,x_k \subset \{0,y_1,...,y_n\} \right\}.
    \end{equation*}
    (\ref{epsilon_cover}) shows that $\mathcal{F}$ is an $\epsilon$ - net for $\mathcal{C}$. Moreover, a simple counting argument gives us that $\abs{\mathcal{F}} \leq (n+1)^k$. Thus $\log \abs{\mathcal{F}} \leq k \log (n+1) \lesssim \frac{\log n}{\epsilon^2}$. Thus,
    \begin{equation*}
        \log \mathcal{N}(\mathcal{C}, \norm{\cdot}, \epsilon) \leq \log \abs{\mathcal{F}} \lesssim \frac{\log n}{\epsilon^2},
    \end{equation*}
    which proves the claim.
\end{proof}
Next we use a duality result. For optimal concentration of measure, we need control of covering numbers for a `dual problem'. 

Let $K \subseteq \R^n$ be a symmetric (i.e. $K=-K$) convex body. We define its polar body $K^{\circ}$ as follows:
\begin{equation*}
    K^{\circ}= \{y \in \R^n: \inp{x}{y} \leq 1, \forall x \in K\}
\end{equation*}

In our context, if $K$ is the unit ball for some norm, $K^{\circ}$ is the unit ball for its dual norm. In order to gain control of the size of the set $ \mathcal{B}:= \{g:\Sigma \rightarrow \mathbb{R} :\norm{g}_{L^2(\Sigma) } \leq 1\}$, we require some input from work on the `duality problem for entropy' in convex analysis. The duality problem asks given two symmetric bodies in $\mathbb{R}^n$ do we have that
\begin{equation}\label{entropydualityconj}
    \log N (K, \epsilon J) \sim \log N (J^{\circ}, \epsilon K^{\circ}) ?
\end{equation}
This problem is open in general (see, for example \cite{Bourgain1989}), but in the case where $K$ is the Euclidean unit ball $B_2$, the answer is affirmative. 
\begin{theorem}[\cite{Artstein2004}]\label{duality}
     Let $J$ be any symmetric convex body in $\mathbb{R}^{n}$, and $B_2$ the unit Euclidean ball. Then there exists an absolute constant $\alpha$ such that 
    \begin{equation*}
    \log \mathcal{N} (J, \epsilon B_2) \sim \log \mathcal{N} (B_2, \alpha \epsilon J^{\circ}).
    \end{equation*}
    \end{theorem}
    Note that $B_2^{\circ}= B_2$, so Theorem \ref{duality} is in fact a special case of (\ref{entropydualityconj}).

As mentioned previously, let $(\alpha_k)_{k=1}^{CR^d}$ be a collection of unit balls which are a finitely overlapping cover of $B_{R}$ in $\mathbb{R}^d$. For the ease of presentation, for the rest of this section, let $n = CR^d$. For each $1 \leq k \leq n$, define the functions $\Tilde{\alpha_k} \in \L^2(\Sigma, \R)$ by $\Tilde{\alpha_k}(\omega) = \int_{\alpha_k} \cos({2\pi \omega \cdot x}) dx$.  If $\mathcal{C}$ is the closed convex hull of of the $\Tilde{\alpha_k}$'s, i.e. $\mathcal{C} = \{\sum_{k=1}^{n} : \sum_{k=1}^{n} \lambda_k \Tilde{\alpha_k} : \sum_{k=1}^{n}\abs{\lambda_k} \leq 1 \}$, then we have by Theorem \ref{maurey} that $\log \mathcal{N}(\mathcal{C}, \norm{\cdot}_{L^2(\Sigma)}, \epsilon) \lesssim \frac{\log n}{\epsilon^2}$. Let $W$ be the subspace in $L^2(\Sigma)$ given by $\mathrm{span} \{ \tilde{\alpha_k} : 1 \leq k \leq n\}$. Since $W$ is a finite dimensional subspace of $L^2(\Sigma)$ of dimension at most $n$, it may be identified with $\R^N$ with for some $N \leq n$ where $W$ inherits the Euclidean structure of $W$ since $L^2(\Sigma)$ is a Hilbert space. Thus the $\Tilde{\alpha_k}'s$ may be identified with vectors in $\R^n$. Using the fact that $\norm{\Tilde{\alpha_k}}_{L^2(\Sigma)} \lesssim 1$, an application of Theorem \ref{duality} yields that that
\begin{equation}\label{dualestimate}
    \log \mathcal{N}(B_2, \epsilon \mathcal{C}^{\circ}) \lesssim \frac{\log n}{\epsilon^2},
\end{equation}
where $B_2$ is the Euclidean ball in $\R^N$. 
\begin{lemma}
    We have
    \begin{equation*}
        \mathcal{C}^{\circ} =\{x \in \R^N: \abs{\inp{x}{\Tilde{\alpha_k}}} \leq 1,  \quad \forall \quad 1 \leq k \leq n\}.
    \end{equation*}
    \begin{proof}
        Let $Y = \{x \in \R^N: \abs{\inp{x}{\Tilde{\alpha_k}}} \leq 1,  \quad \forall \quad 1 \leq k \leq n\}$. If $x \in \mathcal{C}^{\circ}$, then one must have that $\inp{x}{\Tilde{\alpha_k}}$ and $\inp{x}{-\Tilde{\alpha_k}} \leq 1$ for all $1 \leq k \leq n$ (since $\Tilde{\alpha_k}$, $-\Tilde{\alpha_k} \in \mathcal{C}$ for all $k$). Thus, $\abs{\inp{x}{\alpha_k}} \leq 1$ for all $k$, which implies $\mathcal{C}^{\circ} \subseteq Y$. On the other hand if $y \in Y$, and $x \in \mathcal{C}$ is arbitrary, then writing ${x}= \sum_{k=1}^{n}\lambda_k \Tilde{\alpha}_{k}$ for some $\lambda_k$'s for which $\sum_{k=1}^{n}\abs{\lambda_k} \leq 1$ gives that
        \begin{equation*}
            \inp{x}{y} \leq \abs{\inp{x}{y}} \leq \sum_{k=1}^{n} \abs{\lambda_k}\abs{\inp{\Tilde{\alpha_k}}{y}} \leq \sum_{k=1}^{n}\abs{\lambda_k} \leq 1.
        \end{equation*}
        Since $x \in \mathcal{C}$ was arbitrary, this shows that $Y \subseteq \mathcal{C}^{\circ}$, which proves the claim.
    \end{proof}
\end{lemma}
\begin{proposition}
    Define the seminorm $\norm{\cdot}_{\sim}$ on $L^2(\Sigma)$ given by $\norm{g}_{\sim} = \max_{1 \leq k \leq n} \abs{{\inp{g}{\Tilde{\alpha_k}}}}_{L^2(\Sigma)}$. Then
    \begin{equation*}
            \log \mathcal{N}(\mathcal{B}, \norm{\cdot}_{\sim}, \epsilon) \lesssim \frac{\log n}{\epsilon^2}.
    \end{equation*}
    \begin{proof}
        Write $L^2(\Sigma) = W \oplus W^{\perp}$. For any $f \in L^2(\Sigma)$, we may decompose $ f = f_{W} + f_{W^{\perp}}$, and observe that $\norm{f}_{\sim} = \norm{f_{W}}_{\sim}$. Thus the essential information needed to compute the covering number $\mathcal{N}(\mathcal{B}, \norm{\cdot}_{\sim}, \epsilon)$ is restricted to the subspace $W$ (which as we have noted may be identified with $\mathbb{R}^{N}$). If $B_2$ is the unit ball in $\mathbb{R}^n$, then it is clear that
        \begin{equation*}
            \log \mathcal{N}(\mathcal{B}, \norm{\cdot}_{\sim}, \epsilon) = \log \mathcal{N}(B_2, \epsilon \mathcal{C}^{\circ}) \lesssim \frac{\log n}{\epsilon^2}
        \end{equation*}
         by (\ref{dualestimate}).
    \end{proof}
    \begin{remark}
        Note that 
        \begin{equation*}
            \abs{\inp{g}{\Tilde{\alpha_k}}} = \abs{\int_{\Sigma}g(\omega) \Tilde{\alpha_k}(\omega)} = \abs{\int_{\Sigma}\int_{\alpha_k}g(\omega) \cos (2 \pi \omega \cdot x) dx} = \abs{\int_{\alpha_k}\Tilde{E}g(x)dx}.
        \end{equation*}
        by Fubini. Therefore $\norm{g}_{\sim} = \max _{1 \leq k \leq n} \abs{\int_{\alpha_k}\Tilde{E}g(x) dx}$.
    \end{remark}
\end{proposition}
\begin{proposition}[M--T Conjecture is true for $w(x) = \mathbbm{1}_{B_R}(x)$]\label{agmonhormander}
     We have that 
    \begin{equation*}\label{arian}
        \int_{B_R} \abs{Eg(x)}^2 dx \lesssim R \int_{\Sigma}\abs{g(\omega)}^2 d \sigma(\omega).
    \end{equation*}
    \begin{proof}
        This follows from the Agmon-H\"ormander trace inequality (see appendix).
    \end{proof}
\end{proposition}
A straightforward consequence of this is that if $g: \Sigma \rightarrow \R$, then 
    \begin{equation}\label{averagecontrol}
        \sum_{k=1}^{n} \abs{\int_{\alpha_k}\tilde{E}g(x) dx}^2 \lesssim \sum_{k=1}^{n} \int_{\alpha_k}\abs{Eg(x)}^2 dx \lesssim  \int_{B_R}\abs{Eg(x)}^2 dx \lesssim R\int_{\Sigma}\abs{g(\omega)}^2 d \sigma(\omega)
    \end{equation}
    Here the first inequality is an application of Cauchy-Schwarz (using the fact that each $\alpha_k$ has measure $\sim 1$), and the second inequality uses the fact that the $\alpha_k$'s have finite overlap and cover $B_R$, along with (\ref{arian}).

     Choose $0< \delta < 1$ such that $\delta = O(\frac{1}{R})$ so that $\sum_{k=1}^{n} \delta\abs{\int_{\alpha_k}\tilde{E}g(x) dx}^2 \leq \int_{\Sigma}\abs{g(\omega)}^2 d \sigma(\omega)$ for all $g \in L^2(\Sigma, d \sigma)$ and $R \geq 0$. Let $(\delta_i)_{1 \leq i \leq n}$ be i.i.d. selector random variables with $\mathbb{P}(\delta_i = 1)= \delta$. We now begin by establishing some measure concentration inequalities, which will eventually relate to the covering number bounds just established.
\begin{proposition}[A variant of Bennett's inequality \cite{Talagrand1992}, also see appendix]
    Consider a random variable $Z$ with $\abs{Z} \leq 1$ a.e, $\mathbb{E}Z =0$, and $E Z^2 \leq \delta$. Let $(Z_i)_{i=1}^{n}$ be a sequence of i.i.d. copies of $Z$, and let $a= (a_i)_{i=1}^{n}$ be a sequence of real numbers. Then for all $t>0$, we have the following large deviation inequality:
    \begin{equation}\label{Bennett}
        \mathbb{P}\left(\sum_{k=1}^{n}a_kZ_k \geq t \right) \leq \exp (- \frac{t}{\norm{a}_{\infty}} \log \frac{t\norm{a}_{\infty}}{\delta \norm{a}_2^2}).
    \end{equation}
\end{proposition}
\begin{corollary}\label{bennetcorollary}
    Let $(\delta_k)_{k=1}^{n} \sim \mathrm{Bernoulli} (1/R)$ be an i.i.d. collection of selector random variables, and any $I \subseteq \{1,2,3,...,n\}$. Then for any $u \geq 2 \delta \mathrm{card}(I)$,
\end{corollary}
\begin{equation}\label{bennetapplication}
    \mathbb{P}\left(\sum_{k \in I} \delta_k \geq u \right) \leq \exp(-\frac{u}{8} \log \frac{u}{2 \delta \mathrm{card}(I)}).
\end{equation}
 \begin{proof}
First we begin by noting that 
    \begin{equation}\label{zili}
        \mathbb{P}\left(\sum_{k \in I} \delta_k  \geq u \right) = \mathbb{P}\left(\sum_{k \in I} (\delta_k - \delta) \geq u - \delta \mathrm{card}(I) \right) \leq \mathbb{P}\left(\sum_{k \in I} (\delta_k- \delta) \geq \frac{1}{2}u \right)
    \end{equation}
    and now we apply Proposition \ref{Bennett} to the last term of (\ref{zili}) with the sequence $(a_k)_{k=1}^{n}$ where $a_k = 1$ when $k \in I$ and $a_k=0$ otherwise. Thus $\norm{a_k}_{\infty} = 1$ and $\norm{a}_2^2 = \mathrm{card} (I)$, which yields (\ref{bennetapplication}).
 \end{proof}
 We now use a strategy that mimics the argument of Talagrand in \cite{Talagrand1992}, section 3 with some appropriate modifications. 
 First we note that by Cauchy-Schwarz, $\abs{\int_{\alpha_k}\tilde{E}g(x)dx}\lesssim 1$ for each $1 \leq k \leq n$ and $g \in \mathcal{B}$. Without loss of generality, we may assume that $\abs{\int_{\alpha_k}\tilde{E}g(x)dx}\leq 1$ to prove Proposition \ref{mainprop} (by rescaling $\mathcal{B}$ if necessary). 

 \begin{remark}\label{rescalingremark}
 Expanding on the comment above, we will use the following `rescaling' argument frequently. Let 
 \begin{equation*}
     K= \max_{1 \leq i \leq n} \left\{ \sup_{g \in L^2(\Sigma)} \abs{\int_{\alpha_k} \tilde{E}g(x)dx}^2dx\right\},
 \end{equation*}
by Cauchy-Schwarz, it is clear that $K=O(1)$. Now note that 
\begin{equation*}
    \mathbb{E}\sup_{g \in K^{-1}\mathcal{B}} \sum_{k=1}^{n}\delta_k \abs{\int_{\alpha_k}\tilde{E}g(x)dx}^2 = \frac{1}{K^2}\mathbb{E}\sup_{g \in \mathcal{B}} \sum_{k=1}^{n}\delta_k \abs{\int_{\alpha_k}\tilde{E}g(x)dx}^2, 
\end{equation*}

by linearity of the map $g \rightarrow \tilde{E}g$. Thus we may assume that $\abs{\int_{\alpha_k}\tilde{E}g(x)dx}\leq 1$, since we will only lose a constant factor in our bound for the expression (\ref{cos}). We will apply this type of argument frequently, particularly in the next section we rescale $\mathcal{B}$ by a factor $R^{-\epsilon}$, which will lead to an $R^{\epsilon}$ type loss in the relevant bound.
\end{remark}
\textbf{A Dyadic Decomposition:}
For each $k \in \mathbb{N}$, we define 
\begin{equation*}
    l_{j,k}(g) =
    \begin{cases}
         \abs{\int_{\alpha_k}\tilde{E}g(x)}^2 \text{   if  } 2^{-k} \leq \abs{\int_{\alpha_k}\tilde{E}g(x)}^2 \leq 2^{-k+1}\\
        0 \quad \text{otherwise}
    \end{cases}
\end{equation*}
We claim that it in order to prove proposition \ref{mainprop} it suffices to show that
\begin{equation}\label{dyadicbound}
    \mathbb{E}\sup_{g \in \mathcal{B}}\sum_{j=1}^{n} \delta_j l_{j,k}(g) \lesssim \log n
\end{equation}
for each $k \in \mathbb{N}$. Indeed we may write
\begin{align*}
        \mathbb{E} \sup_{\norm{g}_{L^2}(\Sigma) \leq 1 } \sum_{j=1}^{n}\int_{\alpha_k} \abs{Eg(x)}^2dx &= \mathbb{E} \sup_{g \in \mathcal{B}} \sum_{j=1}^{n} \sum_{k=1}^{\infty}\delta_j l_{j,k}(g)\\
        &= \mathbb{E} \sup_{g \in \mathcal{B}} \sum_{j=1}^{n} \sum_{k=1}^{O(\log n)}\delta_j l_{j,k}(g) + \mathbb{E} \sup_{g \in \mathcal{B}} \sum_{j=1}^{n}\sum_{k=O(\log n)}^{\infty}\delta_j l_{j,k}(g)
\end{align*}
The second term may be bounded simply by noting that $l_{k,j}(g) \lesssim 2^{-k}$ for all $g \in \mathcal{B}$. Therefore, 
\begin{equation*}
    \mathbb{E} \sup_{g \in \mathcal{B}} \sum_{j=1}^{n}\sum_{k=1}^{O(\log n)}\delta_j l_{j,k}(g) \lesssim \sum_{j=1}^{n} \sum_{k=O(\log n)}^{\infty} 2^{-k} \lesssim n \cdot 2^{-O( \log n)} \lesssim 1.
\end{equation*}
On the other hand, for first term we have that
\begin{equation*}
    \mathbb{E} \sup_{g \in \mathcal{B}} \sum_{j=1}^{n} \sum_{k=1}^{O(\log n)}\delta_j l_{j,k}(g)
    \leq \sum_{k=1}^{O(\log n)}\mathbb{E} \sup_{g \in \mathcal{B}} \delta_j l_{j,k}(g)
\end{equation*}
and thus if we can show (\ref{dyadicbound}) is true for every $k \in \mathbb{N}$, then proposition \ref{mainprop} will follow immediately. 

We now show that the bound (\ref{dyadicbound}) holds for some fixed $k \in \mathbb{N}$. For any $g \in \mathcal{B}$, we define $I(g) \subseteq \{1,2,...,n\}$ such that $I(g):= \left\{1 \leq i \leq n : \abs{\int \tilde{E}g(x) dx } \geq 2^{-k-1} \right\}$.

\begin{lemma}\label{bennetforg}
    For any $g \in \mathcal{B}$ we have that
    \begin{equation}
        \mathbb{P}\left(\sum_{i \in I(g)} \delta_i \geq 2^{2k+3}t\right) \leq \exp\left(-2^{2k}t \log \log R\right) ,   
        \end{equation}
    for all $t \geq C$, where $C$ is an absolute constant with $C \geq 1$.
\begin{proof}
    First note that by Chebyshev's inequality we have that 
    \begin{equation}\label{cardbound}
        2 \delta \mathrm{card}(I(g)) \leq  2 \cdot \frac{\delta \sum_{j=1}^{n} \abs{\int_{\alpha_j}\tilde{E}g(x)dx}^2}{( 2^{-k-1})^2} \leq \frac{2^{2k+3}}{\log R}.
    \end{equation}
    Now, setting $u=2^{2k+3}t$ where $ t \geq C \geq 1$, we see that in this case $u \geq 2 \delta \mathrm{card}( I(g))$ for all $g \in \mathcal{B}$ and thus Corollary \ref{bennetapplication} applies. Using Corollary 8, (\ref{cardbound}), and the fact that $t \geq$ 1, shows (\ref{bennetforg}).
\end{proof}
\end{lemma}
Define
\begin{equation}\label{alpha_t}
    \alpha_{t} = \mathbb{P}\left(\sup_{g \in \mathcal{B}} \sum_{j=1}^{n} l_{j,k}(g) > 8ct \right).
\end{equation}
Our goal is to show that $\int_{0}^{\infty} \alpha_t dt \lesssim (\log R)^{O(1)}$. Fix $t \geq C$. For any fixed $t$ we define $j_0(t)$ to be the largest natural number such that 
\begin{equation}\label{j_0}
    j_0(t) \exp (-2^{2k}t \log \log R) < \alpha_t.
\end{equation}
The next proposition gives a relationship between the quantity $\alpha_t$ and the covering numbers of $\mathcal{B}$.
\begin{proposition}\label{entropybound}
    For any $ t \geq C$, we have that 
    \begin{equation*}
        j_0(t) \leq \mathcal{N}(\mathcal{B}, \norm{\cdot}_{\sim}, 2^{-k-1}).
    \end{equation*}
\end{proposition}
We prove the above proposition by constructing a sequence of elements $\{g_1, g_2,..., g_{j_0(t)}\} \subset \mathcal{B}$ that are `well spaced' (i.e. they are $2^{-k-1}$- separated in the normed space $(\mathcal{B}, \norm{\cdot}_{\sim})$. The precise construction is given below.

 We construct the sequence $\{g_1, g_2,...,g_{j_0(t)} \} \subset \mathcal{B}$ with the property that for every $u \in \{1,2,..., j_0(t)\}$ there exists some $i_{u,j} \in \{1,2,..., n\}$  such that for each $j < u$ we have
\begin{equation*}
    \abs{\int_{\alpha_{i_{u,j}}} \tilde{E}g_j(x) dx} \geq 2^{-k} \quad \text{but} \quad \abs{\int_{\alpha_{i_{u,j}}}\tilde{E}g_u(x)dx} \leq 2^{-k-1}.
\end{equation*}
Having constructed this sequence we note that for any $m,n$ with $1 \leq m<n \leq j_0(t)$ we have
\begin{align*}
    \norm{g_m -g_n}_{\sim} &= \max_{1 \leq j \leq n} \abs{\int_{\alpha_j} \tilde{E}(g_m-g_n)(x)dx} \\
    &=\max_{1 \leq j \leq n} \abs{\int_{\alpha_j}\tilde{E}g_m(x)dx - \int_{\alpha_j}\tilde{E}g_n(x) dx}\\
    &\geq \abs{\int_{\alpha_{i_{n,m}}}\tilde{E}g_m(x) dx - \int_{\alpha_{i_{n,m}}}\tilde{E}g_n(x) dx}\\
    & \geq \abs{\int_{\alpha_{i_{n,m}}}\tilde{E}g_m(x) dx} - \abs{\int_{\alpha_{i_{n,m}}}\tilde{E}g_n(x) dx}\\
    &\geq 2^{-k}-2^{-k-1}\\
    &=2^{-k-1}.
\end{align*}
Therefore $g_1, g_2, ..., g_{j_0(t)}$ form a $2^{-k-1}$ separated subset of $(\mathcal{B}, \norm{\cdot}_{\sim})$ which implies that $\mathcal{N}(\mathcal{B}, \norm{\cdot}_{\sim}, 2^{-k-1}) \geq j_0(t)$.

\textbf{The Construction:} We construct the $g_i$'s probabilistically and by induction. Set $g_1$ to be the function that is identically $0$ so that $Eg_1 \equiv 0$. Now, having constructed $g_1, g_2,...,g_u$ with the desired properties as given by the construction above, we construct $g_{u+1}$ as follows, using the probabilistic method. Let $\Omega$ be our underlying probability space. For each $v \leq u$ define the events:
\begin{equation*}
    Y_v = \{\omega \in \Omega: \sum_{i \in I(g_v)} \delta_i(\omega) \leq 2^{2k+3}t\}.
\end{equation*}
By Lemma \ref{bennetforg} we have that
\begin{equation*}
    \mathbb{P}(Y_v^{c}) \leq \exp(-2^{2k}t \log \log R).
\end{equation*}
Also define the event
\begin{equation*}
    A= \{\omega \in \Omega : \sup \sum_{j=1}^{n} \delta_j(\omega) l_{j,k}(g) > 8ct\}.
\end{equation*}
By definition $\mathbb{P}(A)=\alpha_t$. Now,
\begin{align*}
    \mathbb{P}\left(\bigcap_{v=1}^{u} Y_v \cap A \right) &= 1 - \mathbb{P}\left(\left(\bigcap _{v=1}^{u}Y_v\right)^{c} \cup A^c\right)\\
    &\geq 1 - \sum_{v=1}^{u} \mathbb{P}(Y_v^c) - \mathbb{P}(A^c)\\
    &= \mathbb{P}(A)- \sum_{v=1}^{u}\mathbb{P}(Y_v^c)\\
    &\geq \alpha_t - u \exp(2^{-2k}\log \log R)\\
    &>0.
\end{align*}
Hence we can find $\omega_0 \in \Omega$ such that
\begin{equation}\label{smallbounddelta}
    \sum_{i \in I(g_v)} \delta_i(\omega_0) \leq 2^{2k+3}t,
\end{equation}
for all $v \leq u$, and 
\begin{equation}\label{probmethodmax}
    \sup_{g \in \mathcal{B}} \sum_{j=1}^{n} \delta_j(\omega_0) l_{j,k}(g) > 2^{2k+3}t.
\end{equation}
As a consequence of (\ref{probmethodmax}) we may find $g_{u+1} \in \mathcal{B}$ such that
\begin{equation*}
    \sum_{j=1}^{n} \delta_j(\omega_0) l_{j,k}(g_{u+1}) > 8t.
\end{equation*}
Now recall that $l_{j,k}(g) \leq c^22^{-2k+2}$ for any $g \in \mathcal{B}$ and moreover $l_{j,k}(g)$ being non zero implies that $l_{j,k}(g)>2^{-2k}$. Therefore, 
\begin{equation}\label{largecard}
    \mathrm{card} \left\{1 \leq j \leq n:  \delta_j(\omega_0)=1, \abs{\int_{\alpha_j}\tilde{E}g_{j+1}(x)dx} > 2^{-k} \right\} >\frac{8t}{2^{-2k}} =  2^{2k+3}t 
\end{equation}
On the other hand, by (\ref{smallbounddelta}) for each $v \leq u$, we see that $\delta_j(\omega_0) = 1$ if and only if $l_{j,k}(g_v) \geq 2^{-k-1}$. Therefore
\begin{equation}\label{smallcard}
    \mathrm{card}\left\{1 \leq j \leq n: \delta_j(\omega_0) =1, \abs{Eg_{v}(x)dx} > 2^{-k-1}\right\} \leq 2^{2k+3}t
\end{equation}
Comparing equations (\ref{largecard}) and (\ref{smallcard}) we can easily infer that each $v \leq u$ we may find $i_{u+1,v} \in \{1,2,...,n\}$ so that
\begin{equation}
    \abs{\int_{\alpha_{i_{u+1,v}}} \tilde{E}g_{u+1}(x)dx} \geq 2^{-k} \quad \text{but} \quad \abs{\int_{\alpha_{i_{u+1,v}}} \tilde{E}g_{v}(x)dx} \leq 2^{-k-1}.
\end{equation}
This closes the inductive argument and concludes the proof of Proposition \ref{entropybound}. 

We now (almost) prove Proposition \ref{mainprop}. We show that 
\begin{equation*}
    \mathbb{E} \sup_{g \in \mathcal{B}} \sum_{j=1}^{n} \delta_j \abs{\int_{\alpha_j} \tilde{E}g(x) dx} \lesssim (\log R)^{O(1)}
\end{equation*}
This is equivalent to showing that 
\begin{equation}\label{goal}
    \int \alpha_t dt \lesssim (\log R)^{O(1)},
\end{equation}
where $\alpha_t$ was defined in (\ref{alpha_t}). Since $j_0(t)$ is the largest integer for which (\ref{j_0}) holds, we have 
\begin{align*}
    \alpha_t &\leq \max \{1, 2j_0(t)\} \exp (-2^{2k+3}t \log \log R)\\ 
    &= \max\{\exp (-2^{2k+3}t \log \log R), 2j_0(t) \exp(-2^{2k+3}t\log \log R) \}.
\end{align*}
If the first term of the max dominates, we can easily deduce (\ref{goal}). We may assume that the second term dominates. Combining the above equation and Proposition \ref{entropybound} we see that
\begin{align*}
    \alpha_t &\leq 2 \cdot \mathcal{N}(\mathcal{B}, \norm{\cdot}_{\sim}, 2^{-k-1})\exp(-2^{2k+3}t \log \log R)\\
    &\leq 2 \cdot \exp( -2^{2k+3} t \log \log R + C_1(d) 2^{2k+2} \log R)
\end{align*}
When $t \geq O(\log R)$, then note that 
\begin{equation*}
    2^{2k+3}t \log \log R - C_1(d) 2^{2k+2}\log R \gtrsim t\log \log R
\end{equation*}.
Now,
\begin{align*}
    \int_{0}^{\infty} \alpha_t dt &\leq \int_{0}^{O(\log R)} \alpha_t dt + \int_{O(\log R)}^{\infty} \alpha_t dt\\
    &\lesssim \log R + \int_{O(\log R)}^{\infty} \alpha_t dt
\end{align*}
and the second integral can be bounded as follows
\begin{align*}
    \int_{O(\log R)}^{\infty} \alpha_t dt &\leq \int_{O(\log R)}^{\infty} \exp(-2^{2k+3}t\log \log R + C_1(d)2^{2k+2}\log R)\\
    &\leq \int_{O(\log R)}^{\infty} \exp (-ct\log \log R)\\
    & \leq \int_{O(1)}^{\infty} \exp(-ct \log \log R) \\
    & \lesssim 1.
\end{align*}
Putting everything together we see that:
\begin{equation*}
    \mathbb{E} \sup_{\norm{g}_{L^2(\Sigma, \R)} \leq 1 } \sum_{k=1}^{CR^d}\abs{\int_{\alpha_k} \tilde{E}g(x)dx}^2 \lesssim (\log R)^{O(1)}.
\end{equation*}
\begin{remark}[Proof of Proposition \ref{mainprop}] The proof of Proposition \ref{mainprop} is easily seen by replacing $\cos(2\pi \omega \cdot x)$ with $\sin(2\pi \omega \cdot x)$ in the preceding argument to note that if $\tilde{\tilde{E}}g(x) \coloneqq \int_{\Sigma}g(\omega)e^{2 \pi i \omega \cdot x}d \sigma(\omega)$, then
\begin{equation*}
     \mathbb{E} \sup_{\norm{g}_{L^2(\Sigma, \R)} \leq 1 } \sum_{k=1}^{CR^d} \delta_k \abs{\int_{\alpha_k} \tilde{\tilde{E}}g(x)dx}^2 \lesssim (\log R)^{O(1)}.
\end{equation*}
Standard manipulations show (unpacking the definition of $Eg$, and splitting $g$ and $Eg$ into real and complex parts) that
\begin{align*}
    &\mathbb{E}\sup_{\norm{g}_{L^2(\Sigma)} \leq 1} \sum_{k=1}^{CR^d}\abs{\int_{\alpha_k} Eg(x)dx}^2 \\ & \lesssim \mathbb{E} \sup_{\norm{g}_{L^2(\Sigma, \R)} \leq 1 } \sum_{k=1}^{CR^d} \delta_k \abs{\int_{\alpha_k} \tilde{E}g(x)dx}^2 + \mathbb{E} \sup_{\norm{g}_{L^2(\Sigma, \R)} \leq 1 } \sum_{k=1}^{CR^d} \delta_k \abs{\int_{\alpha_k} \tilde{\tilde{E}}g(x)dx}^2\\ 
    &\lesssim (\log R)^{O(1)},
\end{align*}
    which proves the claim.
\end{remark}

\section{Proof of Proposition \ref{technicalrefinedprop}}\label{sectionprop9}

In order to upgrade the previous estimate to the proof of Proposition \ref{technicalrefinedprop}, we make use of the uncertainty principle. Roughly speaking, since the Fourier transform of $Eg$ is contained inside a compact set, $Eg$ is roughly constant at scale 1, so $Eg$ is roughly constant on each $\alpha_k$, and therefore we should have:
\begin{equation*}
    \abs{\int_{\alpha_k} Eg(x)dx}^2 \approx \int_{\alpha_k}\abs{ Eg(x)}^2dx
\end{equation*}
We now make this precise, taking into account the proof of Proposition \ref{technicalrefinedprop}. In particular, we need to obtain local constancy estimates so that we can re-run the proof of proposition \ref{mainprop}. Some of the necessary lemmas that come from relatively standard local constancy type estimates are proved in the appendix. For the rest of this section we let $\mathcal{B} \coloneqq \{g: \Sigma \rightarrow \C : \norm{g}_{L^2(\Sigma)} \leq 1\}$.

Let $\omega = (\omega_1,\omega_2,\omega_3,..., \omega_d) \in \R^d$, for each $l > 0$, and every multi-index $(i_1,i_2,...,i_l)$ with $ 1 \leq i_j \leq d$ for all $1 \leq j \leq l$, define $Eg_{i_1,i_2,...,i_l}(x) = \int_{\Sigma}g(\omega)e^{2 \pi i \omega \cdot x}\omega_{i_1}\omega_{i_2}...\omega_{i_l}\sigma(\omega)$. Our first lemma is a technical lemma that more or less allows us to control the average value of $\abs{Eg(x)}$ over a smaller ball $\alpha_k'$ contained in $\alpha_k$ in terms of a specific point $x_k' \in \alpha_k'$.

\begin{lemma}
    Fix $k$  and let $\alpha_k' \subset \alpha_k$ be a ball of radius $\sim R^{-\epsilon}$, so that $\alpha_k' = B(x_k',R^{-\epsilon})$. Then we have that for all $x \in \alpha_k'$ that
    \begin{equation}\label{egsquareusablebound}
            \abs{Eg(x)}^2 \lesssim_{\epsilon, d} \sum_{l=0}^{\frac{10000d}{\epsilon}} \sum_{i_1,...,i_l} \frac{(2 \pi R)^{-2l \epsilon}}{(l!)^2}\abs{Eg_{i_1,...,i_l}(x_k)}^2+O(R^{-10000d}).
        \end{equation}
\end{lemma}
        \begin{remark}
            The key point of this lemma is to control $\abs{Eg(x)}$ by $\abs{Eg_{i_1,...,i_l}(x_k)}$. The latter quantity behaves similarly to the operator $Eg$ and its properties re-running the proof of Proposition \ref{mainprop} with respect to $Eg_{i_1,...,i_l}(x_k)$ will be how we upgrade to the proof of Proposition \ref{technicalrefinedprop}.
        \end{remark}
    \begin{proof}
        We begin by estimating $\abs{Eg(x)}$ on $\alpha_k'$. Suppose $x \in \alpha_k'$. We have 
        \begin{align*}
            \abs{Eg(x)} &= \abs{\int_{\Sigma} g(\omega)e^{2\pi i \omega \cdot x} d \sigma(\omega)}\\
            &= \abs{\int_{\Sigma} g(\omega)e^{2\pi i \omega \cdot (x-x_0)} e^{2\pi i \omega \cdot x_0} d\sigma(\omega)}
        \end{align*}
        We may write $e^{2 \pi \omega \cdot (x-x_k)} = \sum_{k=0}^{N}\frac{(2\pi i \omega \cdot (x-x_k))^l}{l!} + O(R^{-\epsilon(N+1)})$ for some sufficiently large $N$ (to be chosen later), and thus plugging this estimate into the above equation yields that
        \begin{align}\label{epsloss}
            \abs{Eg(x)} &\lesssim \sum_{k=0}^{N}\frac{1}{l!}\abs{\int_{\Sigma} g(\omega) e^{2 \pi i \omega \cdot x_k} (2 \pi i\omega\cdot (x-x_k))^{l} d \sigma({\omega})} +O(R^{-N\epsilon})\\
            & \lesssim \sum_{k=0}^{N} \sum_{i_1,i_2,...i_l}\frac{(2\pi R)^{-\epsilon l}}{l!}\abs{\int_{\Sigma} g(\omega) e^{2 \pi i \omega \cdot x_k} \omega_{i_1} \omega_{i_2}...\omega_{i_l} d \sigma({\omega})} +O(R^{-N\epsilon}). 
        \end{align}
        Here $\omega\in \R^d$ is written as $\omega = (\omega_1,\omega_2,...,\omega_d)$, and $(i_1,...,i_l)$ is a multi-index that ranges over all possible $k-$tuples of $\{1,2,...,d\}$. The above estimate follows from expanding the integrals, and the triangle inequality along with the fact that $\abs{x-x_k} \lesssim R^{-\epsilon}$ for all $x \in \alpha_k'$.
        
        Using the Cauchy-Schwarz inequality, we obtain a bound for $\abs{Eg(x)}^2$:
        \begin{equation}\label{egsquare}
            \abs{Eg(x)}^2 \lesssim \left(N\sum_{l=0}^{N}d^l \right) \cdot \left( \sum_{l=0}^{N} \sum_{i_1,i_2,...,i_l} \frac{(2\pi R)^{-2l \epsilon}}{(l!)^2} \abs{Eg_{i_1,...,i_l}(x_k)}^2\right) + O(R^{-2N\epsilon}).
        \end{equation}

        Set $N= \frac{10000d}{\epsilon}$. In this case the term $N \sum_{l=0}^{N}d^l$ is a constant that only depends on $d$ and $\epsilon$. In light of this we rewrite (\ref{egsquare}) as
        \begin{equation}\label{egsquareusablebound}
            \abs{Eg(x)}^2 \lesssim_{\epsilon, d} \sum_{l=0}^{\frac{10000d}{\epsilon}} \sum_{i_1,...,i_l} \frac{(2 \pi R)^{-2l \epsilon}}{(l!)^2}\abs{Eg_{i_1,...,i_l}(x_k)}^2+O(R^{-10000d})
        \end{equation}
        as desired.
        \end{proof}
        
        We now make the following claim. Let $\{x_{k}'\}_{k=1}^{N}$ be a $R^{-\epsilon}$ separated subset of $B_R$. Then for any fixed $(i_1,i_2,...,i_l)$ with $1 \leq i_j \leq d$, we have the following Proposition, which is quite similar to Proposition \ref{mainprop}.
        
        \begin{proposition}\label{mainproppointwise}
            \begin{equation*}
               \mathbb{E} \sup_{g \in B}  \sum_{k=1}^{N} \delta_k \int_{\alpha_{k}'} \abs{Eg_{i_1,i_2,...,i_l}(x_k')}^2dx \lesssim_{\epsilon} R^{\epsilon}.
            \end{equation*}
            \end{proposition}
            \begin{proof}
                First note that since the expression for $Eg_{i_1,i_2,...,i_l}(x_k')$ does not depend on the spatial variable $x$, by Fubini we have that $\abs{\int_{\alpha_{k}'} Eg_{i_1,i_2,...,i_l}(x_k')}^2 \sim R^{-\epsilon d} \abs{Eg_{i_1,i_2,...,i_{l}}(x_k')}^2 \sim \int_{\alpha_k'}\abs{Eg_{i_1,i_2,...,i_{l}}(x_k')}^2dx$. Therefore the problem is reduced to showing the equivalent statement that
                \begin{equation}\label{modifiedmainprop}
                    \mathbb{E} \sup_{g \in B}  \sum_{k=1}^{N} \delta_k \abs{\int_{\alpha_{k}'} Eg_{i_1,i_2,...,i_l}(x_k')dx}^2 \lesssim_{\epsilon} R^{\epsilon}.
                \end{equation}
                The key point is that one can control this expression in the same way one proves Proposition \ref{mainprop}. In particular, one shows that (\ref{modifiedmainprop}) is true by modifying the proof of Proposition \ref{mainprop} in the following ways:
                \begin{enumerate}
                    \item The role of $\tilde{\alpha_k}(\omega)$ is played by  (real and imaginary parts of) function  the $\tilde{\alpha_k'}(\omega) = \int_{\alpha_k'}e^{2 \pi i \omega \cdot x_{k}'} \omega_{i_1} \omega_{i_2} ... \omega_{i_l} dx$. Note in particular that $\norm{\tilde{\alpha_k'}(\omega)}_{L^2(\Sigma)} \lesssim 1$.
                    \item A suitable analogue of (\ref{averagecontrol}) is needed. This is given by Lemma \ref{poissonsummationsub} in the appendix. In particular one has that when $\norm{g}_{L^2(\Sigma)} \leq 1$, then
                    \begin{align*}
                        \sum_{k'}\abs{Eg_{i_1,i_2,...,i_l}(x_k')}^2 &\lesssim_{\epsilon} \int_{B_{2R}} \abs{Eg_{i_1,i_2,...,i_l}(x)}^2dx + R^{-1000}\norm{g}_{L^2(\Sigma)}^2\\
                        &\lesssim  R^{1+ \epsilon d} \int_{\Sigma}\abs{g(\omega)}^2 d\sigma(\omega) + R^{-1000}\norm{g}_{L^2(\Sigma)}^2\\
                        &\lesssim_{\epsilon} R^{1+\epsilon d}
                    \end{align*}
                    where the second inequality follows from Proposition \ref{agmonhormander}.
                    \item The other aspects of the proof mimic the proof of Proposition \ref{mainprop}. In particular our definition of the selector random variables $\delta_k$ does not change, however, in order to obtain the $\lesssim_{\epsilon} R^{\epsilon}$ factor in (\ref{modifiedmainprop}), one can rescale $\mathcal{B}$ by a factor of $R^{-\epsilon}$ (see Remark \ref{rescalingremark}), and proceed in the same way we prove Proposition \ref{mainprop}.
                \end{enumerate}
            \end{proof}
            We are now ready to prove the main technical theorem of this paper:
                
            \begin{proof}[Proof of Proposition \ref{technicalrefinedprop}:]
            Recall that we need to show that 
            \begin{equation*}
                    \mathbb{E} \sup_{g \in \mathcal{B}} \sum_{k=1}^{CR^d} \delta_k \int_{\alpha_k} \abs{Eg(x)}^2dx \lesssim_{\epsilon, d} R^{\epsilon}
                \end{equation*}
                Decompose each $\alpha_k$ into $ \sim R^{\epsilon}$ balls $\{\alpha_{k,j}\}_{j \in J}$, with $\mathrm{card}(J) \sim R^{\epsilon}$ let $\{x_{j,k}\}_{j \in J}$ be the centers of these balls. We can now write
                \begin{align}
                    \mathbb{E} \sup_{g \in \mathcal{B}} \sum_{k=1}^{CR^d} \delta_k \int_{\alpha_k} \abs{Eg(x)}^2dx &= \mathbb{E} \sup_{g \in \mathcal{B}} \sum_{k=1}^{CR^d} \sum_{j \in J} \delta_k \int_{\alpha_{k,j}}\abs{Eg(x)}^2dx\\ 
                    & \leq \sum_{j \in J} \mathbb{E} \sup_{g \in \mathcal{B}} \sum_{k=1}^{CR^d} \delta_k \int_{\alpha_{k,j}}\abs{Eg(x)}^2dx \label{sumj}
                \end{align}
                We estimate the inner sum for each $j$ by using the bounds in (\ref{egsquareusablebound}). Fix $j \in J$, we have
                \begin{align*}
                    &\mathbb{E} \sup_{g \in \mathcal{B}} \sum_{k=1}^{CR^d} \delta_k \int_{\alpha_{k,j}}\abs{Eg(x)}^2dx \lesssim_{\epsilon, d} \\
                    &\mathbb{E} \sup_{g \in \mathcal{B}} \sum_{k=1}^{CR^d} \delta_k \sum_{l=0}^{\frac{10000d}{\epsilon}} \sum_{i_1,i_2,...,i_l} \frac{(2\pi R)^{-2\epsilon l}}{(l!)^2}\int_{\alpha_{k,j}}\abs{Eg_{i_1,i_2,...,i_l} (x_{j,k})}^2dx + O(R^{-1000})
                \end{align*}
                Interchanging the inner two summands with the beginning summand in the previous expression, along with interchanging the $\sup$ with positive summands yields that the previous expression can be bounded by:
                \begin{equation*}
                    \sum_{l=0}^{\frac{Cd}{\epsilon}} \sum_{i_1,i_2,...,i_l} \frac{(2\pi R)^{-2\epsilon k}}{(l!)^2
                    }\mathbb{E} \sup_{g \in \mathcal{B}} \sum_{k=1}^{CR^d} \delta_k \int_{\alpha_{k,j}}\abs{Eg_{i_1,i_2,...,i_l} (x_{j,k})}^2dx +O(R^{-1000})
                \end{equation*}
                Proposition \ref{modifiedmainprop} gives us a good estimate for the inner term (noting that $\{x_{j,k}\}$ is a $\sim R^{-\epsilon}$ separated subset of $B_R$ for fixed $j$):
                \begin{equation*}
                    \mathbb{E} \sup_{g \in \mathcal{B}} \sum_{k=1}^{CR^d} \delta_k \int_{\alpha_{k,j}}\abs{Eg_{i_1,i_2,...,i_l} (x_{j,k})}^2dx \lesssim_{\epsilon, d} R^{d\epsilon}.
                \end{equation*}
                We also have a crude bound for the outer sum (by bounding the numerator by $1$).
                \begin{equation*}
                    \sum_{l=0}^{\frac{10000d}{\epsilon}} \sum_{i_1,i_2,...,i_l} \frac{(2\pi R)^{-2\epsilon k}}{(l!)^2
                    } \lesssim_{\epsilon,d} 1.
                \end{equation*}
                Combining this with (\ref{sumj}), and noting that $\mathrm{card}(J) \sim R^{\epsilon}$ completes the proof.
            \end{proof}
        
    \section{Probabilistic Estimates}\label{probestimatesection}
    \subsection{Proof of the Main Theorem (Theorem \ref{maintheoreminformal})}
          We now show that our construction of (random) weights has the property that they have large mass and low tube occupancy, completing the proof of Theorem \ref{maintheoreminformal}.
\begin{lemma}\label{bernoullilargedevbound}
    Given our definition of $\{\delta_k\}$, we define the (random) set $S = \{k: \delta_k =1\}$. We define the (random) weight $w$ as $\sum_{k \in S} \delta_k  \mathbbm{1}_{\alpha_k} $. With this, with high probability the following two properties hold:
    \begin{enumerate}
        \item $\norm{w}_{L^1(B_R)} \sim R^{d-1}$
        \item $\sup_{T \in \mathbb{T}} w(T) \lesssim_{\epsilon} R^{\epsilon}$
    \end{enumerate}
    \begin{proof}
        Property 1 follows from the fact that $\mathbb{E}(\abs{S}) \sim R^{d-1}$. Focusing on property 2, we first note that there are $ \sim R^{2(d-1)}$ distinct 1 - tubes covering $B_R \subseteq \R^d$, and $\sim R$ many unit balls $\alpha_k \subset T$. We find that given a 1-tube $T$, we can estimate the probability that $w(T)$ exceeds some threshold $t$, by the standard Chernoff bound for large deviation estimates for Binomial random variables (see \cite{vershynin2018high}, Theorem 2.3.1).
        \begin{equation*}
            \mathbb{P}(w(T) > t) = \mathbb{P}\left(\sum_{\alpha_k \subset T}   \delta_k > t \right) \lesssim \left(\frac{C}{t}\right)^{t},
        \end{equation*}
        therefore,
        \begin{align*}
            \mathbb{P}\left(\sup_{T \in \mathbb{T}} w(T)> \log R \right) &= \mathbb{P}\left( \bigcup_{T \in \mathbb{T}} \{w(T)>\log R \}\right)\\ 
            & \leq \sum_{T \in \mathbb{T}}\mathbb{P}\left(w(T) > \log R \right)\\
            &\lesssim R^{2(d-1)}\left(\frac{C}{\log R}\right)^{\log R} \rightarrow 0.
        \end{align*}
        Thus, with very high probability, $\mathbb{P}(\sup_{T \in \mathbb{T}} w(T) \leq \log R)$ which establishes property 2. 
        \end{proof}
        \end{lemma}
        To summarize, with high probability, our (randomly constructed) weight $w$ has the following properties:
        \begin{enumerate}
            \item $\int_{B_R}\abs{Eg(x)}^2 w(x) dx \lesssim_{\epsilon,d} R^{\epsilon}\int_{\Sigma}\abs{g(\omega)}^2 d\sigma$, uniformly for all $g \in L^2(\Sigma)$.
            \item $\sup_{T \in \mathbb{T}} w(T) \lesssim_{\epsilon} R^{\epsilon}$.
            \item $\norm{w}_{1} \sim R^{d-1}$.
        \end{enumerate}
        which concludes the proof of Theorem \ref{maintheoreminformal} that with high probability our weights satisfy sharp Mizohata-Takeuchi type estimates (up to an $R^{\epsilon}$ loss, which are now known to be the best possible Mizohata-Takeuchi type estimate one can hope for, see \cite{cairo2025counterexample}).
        
    \subsection{Comparison with Carbery's Random Weight Model}\label{carberymodelrandomness}
    
    In \cite{carbery2009large}, Carbery comes up with a different random construction of a weight $\tilde{w}$ with $supp(\tilde{w}) \subseteq B_R$, that is different than ours. In this section we show that the relevant Mizohata-Takeuchi type estimates that hold for our random construction also hold for Carbery's construction. 

    Carbery's combinatorial construction of a random weight $\tilde{w}$ is (essentially) as follows:
    \begin{enumerate}
        \item Among all $\{\alpha_k\}_{k \in [CR^d]}$,  (recall here that the $\alpha_k$'s are a collection of unit balls covering $B_R$), choose $R^{d-1}$ balls $\tilde{w}_1,\tilde{w}_2,...\tilde{w}_{R^d}$ uniformly at random (without replacement).
        \item Set $\tilde{w}(x) = \sum_{j=1}^{R^{d-1}}\mathbbm{1}_{\tilde{w}_j}(x)$
    \end{enumerate}
    With this construction, in \cite{carbery2009large} Carbery showed that with high probability $ \geq \frac{1}{2}$, that $\sup_{T \in \mathbb{T}}w(T) \lesssim \log R$, thus constructing a large collection of weights that have large mass and low tube occupancy. We note that one of the key reasons for choosing to construct these random weights $w$ in this particular way is that the paper \cite{carbery2009large} attempts to analyze how large $\norm{w}_{1}$ can be if $\sup_{T \in \mathbb{T}}w(T) \sim 1$ (this is related to an open problem known as the $N-$ set occupancy problem; see, for example, \cite{demeter2024nset}). The result that $\sup_{T \in \mathbb{T}}w(T) \lesssim \log(R)$ for this particular random construction is only a special case of some more refined analysis\footnote{In fact, using the same combinatorial Bernoulli trial model one can construct weights for which $\norm{w}_1 \lesssim R^{d-1} \log R$ and $\sup_{T \subseteq \mathbb{T}}w(T) \lesssim \log R$. Our analysis in the forthcoming theorem does not change much with the extra $\log R$ factor added to the mass of the weight, and since we are interested in proving sharp Mizohata-Takeuchi type estimates up to an $R^{\epsilon}$ loss, we consider the model with $\norm{w}_1 \sim R^{d-1}$. } in \cite{carbery2009large}. 
    We prove the following: 
    \begin{proposition}\label{carberycouplingprop}
        Let $\tilde{w}$ be defined as above, then\footnote{We would like to acknowledge Open AI's Chat GPT o3 mini model for some useful suggestions in proving this proposition.}
        \begin{equation}\label{carberycouplingfinal}
            \mathbb{E}\sup_{\norm{g}_{{L^2}(\Sigma) \leq 1}} \int_{B_R} \abs{Eg(x)}^2\tilde{w}(x)dx \lesssim R^{\epsilon}.
        \end{equation}
        \begin{proof}
            The main point here is that while our random weight $w$ and Carbery's $\tilde{w}$ are constructed differently, there are very few instances where a ball unit ball $\alpha_k \subseteq B_R$ is chosen twice (or more than a constant number of times). Therefore the distributions of the weights are more or less comparable.

            Write $\tilde{w}(x) = \sum_{k=1}^{CR^d}\tilde{w_k}\mathbbm{1}_{\alpha_k}(x)$, where each $\tilde{w_k} \sim \text{Binomial}(R^{d-1}, \frac{1}{CR^{d}})$, and the $\tilde{w_k}$ are mutually independent. We may assume that $2d \leq R$ (if this is not the case, the quantity (\ref{carberycouplingfinal}) is otherwise bounded by constant depending only on $d$ which is acceptable). Let $\overline{w_k}= \min\{\tilde{w_k}, 2d\}$. Now we write (\ref{carberycouplingfinal}) as\begin{align*}&\mathbb{E}\sup_{\norm{g}_{L^2(\Sigma)}}\sum_{k}\tilde{w_k}\int_{\alpha_k}\abs{Eg(x)}^2dx \\ & \leq \mathbb{E}\sup_{\norm{g}_{L^2(\Sigma)} \leq 1}\sum_{k}\overline{w_k}\int_{\alpha_k}\abs{Eg(x)}^2dx + \mathbb{E} \sup_{\norm{g}_{L^2(\Sigma)} \leq 1}\sum_{k: \tilde{w_k}>2d} (\tilde{w_k}-2d) \int_{\alpha_k}\abs{Eg(x)}^2dx
            \end{align*}
            We estimate the the two terms separately. For the first term, let $S = \{k: \tilde{w_k}\geq 1\}$. Note that for any $k$, $\mathbb{P}(k \in S) \leq \mathbb{E}(\tilde{w_k}) = O(\frac{1}{R})$ by Markov's inequality. Moreover, the events $A_k =\{k \in S\}$ are mutually independent for all $k$, and thus the random variables $\mathbbm{1}_{A_k}$ are i.i.d. Bernoulli random variables with mean $O(\frac{1}{R})$. Now the first term above can be estimated as follows: 
            \begin{align*}
                \mathbb{E}\sup_{\norm{g}_{L^2(\Sigma) \leq 1}}\sum_{k}\overline{w_k}\int_{\alpha_k}\abs{Eg(x)}^2dx &\leq 2d \cdot \mathbb{E}\sup_{\norm{g}_{L^2(\Sigma) \leq 1}}\sum_{k} \mathbbm{1}_{A_k}\int_{\alpha_k}\abs{Eg(x)}^2dx \\
                &\sim_{d} \mathbb{E}\sup_{\norm{g}_{L^2(\Sigma) \leq 1}}\sum_{k} \delta_k \int_{\alpha_k}\abs{Eg(x)}^2dx \\
                &\lesssim_{\epsilon, d} R^{\epsilon}.
            \end{align*}
            where the last line follows from our main proposition, Proposition \ref{modifiedmainprop}. We turn to estimating the second term. By the Agmon--H\"ormander trace inequality we have the crude estimate that $\int_{\alpha_{k}}\abs{Eg(x)}^2dx \leq \int_{B_R}\abs{Eg(x)}^2dx \lesssim R \norm{g}_{L^2(\Sigma)}^2$. As a consequence,
            \begin{align*}
                &
                \\
                & \leq R \cdot \mathbb{E}\left(\sum_{k:w_k>2d} (\tilde{w_k}-2d)\right)\\
                &\leq R \cdot \mathbb{E}\left(\sum_{k: \tilde{w_k}>2d}\tilde{w_k}\right).
            \end{align*}
            
            We turn to estimating the term $\mathbb{E}\left(\sum_{k: \tilde{w_k} >2d} \tilde{w_k}\right)$, which by linearity of expectation can be rewritten as $\sum_{k=1}^{CR^d}\mathbb{E}(\tilde{w_k} \mathbbm{1}_{\{\tilde{w_k}>2d\}})$. Estimating one of the terms in this sum we have
            \begin{align*}
                \mathbb{E}(\tilde{w_k}\mathbbm{1}_{\{\tilde{w_k}>2d\}}) &= \sum_{j=2d+1}^{R^{d-1}}j\mathbb{P}(\tilde{w_k}=j)\\ 
                &=\sum_{j=2d+1}^{R^{d-1}} j\binom{R^{d-1}}{j}\left(\frac{1}{CR^d}\right)^j\left( 1-\frac{1}{CR^d}\right)^{R^{d-1}-j}\\
                &\lesssim \sum_{j=2d+1}^{R^{d-1}}j \left(\frac{R^{d-1}}{j}\right)^{j}\left(\frac{1}{CR^{d}}\right)^j\\
                & \lesssim \sum_{j=2d+1}^{R^{d}-1} \frac{1}{j^{j-1}R^{j}}\\
                &= \frac{1}{R^{2d}} \sum_{j=2d+1}^{R^d-1}\frac{1}{j^{j-1}R^{j-2d}}\\
                & \lesssim \frac{1}{R^{2d}}.
            \end{align*}
            Adding all of these terms together we get that
            \begin{equation*}
                \mathbb{E}\left(\sum_{k: \tilde{w_k} >2d} w_k \right) = \sum_{k=1}^{CR^d}\mathbb{E}(\tilde{w_k} \mathbbm{1}_{\{\tilde{w_k}>2d\}}) \lesssim CR^{d}\cdot\frac{1}{R^{2d}} \sim \frac{1}{R^d}.
            \end{equation*}
            As a consequence, 
            \begin{equation*}
                \mathbb{E} \sup_{\norm{g}_{L^2(\Sigma)} \leq 1} \sum_{k: \tilde{w_k}>2d} (\tilde{w_k}-2d)\int_{\alpha_k}\abs{Eg(x)}^2dx \lesssim \mathbb{E}\left(\sum_{k: \tilde{w_k}> 2d } \tilde{w_k} \right) \lesssim R \cdot \frac{1}{R^d} \lesssim 1.
            \end{equation*}
            Combined with our previous result this proves Theorem \ref{carberycouplingfinal}.
        \end{proof}
    \end{proposition}

\newpage        
\section{Appendix}
In this appendix we record some supplementary lemmas and facts that are used in this paper, both in probability and Fourier analysis.
\subsection{Some Useful Probability Results}
\begin{lemma}\label{indepexpectation}
    Let $Z_1$ and $Z_2$ be two independent and identically distributed $\mathcal{H}$- valued random variables that take finitely many values with mean $0$. Then, $\mathbb{E}\inp{Z_1}{Z_2} = 0$.
    \begin{proof}
        Fix any $w \in \mathcal{H}$. We first show that $\mathbb{E}\inp{Z_1}{w} = 0$ for all $w \in \mathcal{H}$. Since $Z_1$ has finite support we may assume that $Z_1$ has the following distribution:
        \begin{align*}
            \mathbb{P}(Z_1 = x_k) = \alpha_k
        \end{align*}
        for some $x_1, x_2,...,x_n \in \mathcal{H}$ and $\alpha_1,...,\alpha_n \in \R^{+}$ with $\sum_{k=1}^{n}\alpha_k = 1$. Moreover, since $\mathbb{E}Z_1 = 0$, we have that $\sum_{k=1}^{n}\alpha_k x_k = 0$. Now the ($\mathbb{R}$ - valued) random variable  $\inp{Z_1}{w}$ has the distribution given by:
        \begin{equation*}
            \mathbb{P}(\inp{Z_1}{w}=\inp{x_k}{w}) = \alpha_k.
        \end{equation*}
        Thus, $\mathbb{E} \inp{Z_1}{w} = \sum_{k=1}^{n}\inp{x_k}{w}\mathbb{P}(\inp{Z_1}{w} = \inp{x_k}{w}) = \sum_{k=1}^{n} \alpha_k \inp{x_k}{w} = \inp{\sum_{k=1}^{n}\alpha_k x_k}{w} = \inp{0}{w} = 0$. 
    We now prove the lemma by conditioning on $Z_2$ by using the tower property for expectations.
    \begin{align*}
        \mathbb{E}\inp{Z_2}{Z_1} &= \mathbb{E}_{Z_1} \mathbb{E}_{Z_2} (\inp{Z_1}{Z_2} | Z_2)\\
        &= \mathbb{E}_{Z_2} (0)\\
        &=0
    \end{align*}
    The first line follows from the tower property of conditional expectations and the independence of $Z_1$ and $Z_2$. The second line follows from the fact that $\mathbb{E}\inp{Z_1}{w} =0$ for every $w \in \mathcal{H}$. This proves the claim.
    \end{proof}
\end{lemma}
The following result first appeared in \cite{Talagrand1992}, but we include a short proof of this here for completeness.
\begin{proposition}[Variant of Bennett's inequality \cite{Talagrand1992}]
Consider a random variable $Z$ with $\abs{Z} \leq 1$ a.e, $\mathbb{E}Z =0$, and $E Z^2 \leq \delta$. Let $(Z_i)_{i=1}^{n}$ be a sequence of i.i.d. copies of $Z$, and let $a= (a_i)_{i=1}^{n}$ be a sequence of real numbers. Then for all $t>0$, we have the following large deviation inequality:
    \begin{equation*}
        \mathbb{P}\left(\sum_{k=1}^{n}a_kZ_k \geq t \right) \leq \exp (- \frac{t}{\norm{a}_{\infty}} \log \frac{t\norm{a}_{\infty}}{\delta \norm{a}_2^2}).
    \end{equation*}
\end{proposition}
\begin{proof}
We have by Chebyshev's inequality that for all $ \lambda> 0$:
\begin{align*}
    \mathbb{P}\left(\sum_{k=1}^{n}a_kZ_k \geq t \right) &= \mathbb{P}\left(\exp(\sum_{k=1}^{n}\lambda a_k Z_k) \geq \exp(\lambda t) \right)\\ 
    &\leq e^{-\lambda t}\prod_{k=1}^{n}\mathbb{E}\left(\exp(\lambda a_k Z_k)\right)\\
\end{align*}
Now observe that the inequality
\begin{equation}
    e^{x} \leq 1+x + \frac{1}{2}x^2e^{\abs{x}}
\end{equation}
is true for all $x \in \R$ (this becomes clear if you expand both sides as power series).
Thus,
\begin{equation*}
    \mathbb{E}(\exp(\lambda a_k Z_k)) \leq 1+ \frac{1}{2}\lambda^2 a_k^2 \delta e^{\lambda \abs{a_k}} \leq \exp\left(\frac{1}{2}\lambda^2 a_k^2 e^{\lambda \abs{a_k}}\right).
\end{equation*}
Plugging this into our original estimate gives us that
\begin{align*}
    \mathbb{P}\left(\sum_{k=1}^{n}a_kZ_k \geq t \right) &\leq \exp\left(-\lambda t + \frac{\lambda^2 \delta \norm{a}_2^2}{2}e^{\lambda \norm{a}_{\infty}} \right)\\
    &\leq \exp\left(-\lambda t + \frac{\lambda \delta \norm{a}_2^2}{2\norm{a}_{\infty}}e^{2\lambda \norm{a}_{\infty}} \right)
\end{align*}
(The last step is justified by noting that $x \leq e^x$ when $x>0$ so that $\frac{e^{\lambda \norm{a}_{\infty}}}{\norm{a}_{\infty}} \geq 1$.)
Setting
\begin{equation*}
    \lambda=\frac{1}{2 \norm{a}_{\infty}}\log\frac{t\norm{a}_{\infty}}{\delta \norm{a}_2^2},
\end{equation*}
proves the claim.
\end{proof}
\newpage
\subsection{Some Useful Results from Fourier Analysis}
We begin by collecting some standard `local constancy lemmas' that are used to supplement to proof of Proposition \ref{technicalrefinedprop}.
\begin{lemma}[Convolution Lemma]\label{localconstancyconv} 
        For all $x_0 \in \mathbb{R}$ we have $Eg_{i_1,...i_k} * \eta (x_0)$ where $\eta$ is a Schwartz function with the property that
        \begin{enumerate}
            \item $\hat{\eta} \equiv 1$ on $B(0,1) \subseteq \R^d$.
            \item $\hat{\eta} \equiv 0$ on $\R^d \setminus B(0,2)$
            \item $\eta(x) \leq C_N(1+\abs{x})^{-N}$ for all $N \in \mathbb{N}$.
        \end{enumerate}
        \end{lemma}
        \begin{proof}
            The proof of the first two items is a straightforward consequence of the convolution property for Fourier transforms. Letting $\eta$ be a standard bump function that satisfies properties 1 and 2 above, we have
            \begin{align*}
                Eg_{i_1,i_2...,i_k} * \eta (x_0) &= \int_{\R^d} Eg(y) \eta (x_0-y)dy \\
                &=\int_{\R^d} g(\omega) \omega_{i_1}\omega_{i_2}...\omega_{i_k}  \int_{\Sigma} e^{2 \pi i \omega \cdot y} \eta (x_0-y) dy d\sigma(\omega)\\
                &= \int_{\mathbb{R^d}} g(\omega)  
                e^{2 \pi i \omega \cdot x_0}\omega_{i_1}\omega_{i_2}...\omega_{i_l} \int_{\Sigma}e^{- 2 \pi i \omega \cdot u} \eta(u) du d\sigma (\omega)\\
                &= \int_{\Sigma} g(\omega)  
                e^{2 \pi i \omega \cdot x_0}\omega_{i_1}\omega_{i_2}...\omega_{i_l} \hat{\eta}(\omega) d \sigma (\omega)\\
                &= \int_{\Sigma} g(\omega)  
                e^{2 \pi i \omega \cdot x_0}\omega_{i_1}\omega_{i_2}...\omega_{i_l} d \sigma (\omega)\\
                &= Eg_{i_1,i_2,...i_l}(x_0)
            \end{align*}
            where the third lines follows from Fubini's theorem, fifth line follows from the fact that $\hat{\eta}(\omega) \equiv 1$ for all $\omega \in \Sigma$, since $\Sigma \subseteq B(0,1)$. 

            The third property follows the fact $\eta$ is a Schwartz function (by virtue of being the Fourier transform of a Schwartz function.)
            \end{proof}

            \begin{lemma}[Local Constancy Lemma]\label{weightineq}
                For any $x_i \in B_R$ we have $\abs{Eg_{i_1,i_2,...i_l}(x_i)}^2 \lesssim \norm{Eg_{i_1,i_2,...,i_l}}^2_{L^{2}({w_{x_i}})}$, where $w_{x_i}: \mathbb{R}^d \rightarrow [0, \infty)$ is a weight with the property that $w_{x_i}(y) \leq C_N(1 + dist (y,x_i))^{-N}$ for each $N \in \mathbb{N}$.
            \end{lemma}
            \begin{proof}
                Using Lemma \ref{localconstancyconv} write
                \begin{align*}
                    \abs{Eg_{i_1,i_2,...,i_l}(x_i)}^2 &= \abs{Eg_{i_1,i_2,...,i_l} * \eta(x_i)}^2\\
                    &=\abs{\int_{\R^d}Eg_{i_1,i_2,...,i_l}(y) \eta(x_i-y)dy}^2\\
                    & \leq \int_{\R^d} \abs{Eg_{i_1,...,i_l}(y)}^2\abs{\eta(x_i-y)}dy \int_{\mathbb{R}^d}\abs{\eta(x_i-y)}dy\\
                    & \lesssim \int_{\R^d} \abs{Eg_{i_1,i_2,...,i_l}(y)}^2 \abs{\eta (x_i-y)} dy
                \end{align*}
                where the third line follows by the Cauchy-Schwarz inequality. Setting $w_{x_i}(y) := \abs{\eta(x_i-y)}$, we see that as a consequence of Property 3 of Lemma \ref{localconstancyconv}, that $w_{x_i}(y) \leq C_N(1+ dist(x_i, y))^{-N}$.
            \end{proof}
            \begin{lemma}[Pointwise weight inequalities]\label{ptwise}
                Let $\{x_i\}_{i \in I}$ be a collection of points in $B_{R}$ that are $R^{- \epsilon}$-seperated. Then for any $y \in B_{2R}$, we have the pointwise weight inequality:
                \begin{equation*}
                    \sum_{i \in I}w_{x_i}(y) \lesssim_{\epsilon, d } R^{\epsilon d}.
                \end{equation*}
                \end{lemma}
                \begin{proof}
                    Fix any $y \in B_{2R}$. for each $k \in \N$, let $C_{k}= \{x \in \R^d: 2^{k-1}R^{-\epsilon}\leq dist(x,y) < 2^kR^{-\epsilon}\}$. Note that we have $\mathrm{vol}(C_k) \sim_{d} 2^{kd}R^{-d \epsilon}$. Since $I$ is a maximal $R^{-\epsilon}$ separated set, $\mathrm{card}({C_k \cap I}) \lesssim_{d} 2^{kd}$. Thus
                    \begin{align*}
                        \sum_{i \in I} w_{i}(y) &\lesssim \sum_{k \in \N } \sum_{i \in C_k \cap I} w_{i}(y)\\
                        &\leq C_N \sum_{k \in \N} 2^{kd} \cdot (2^k R^{-\epsilon})^{-N}
                    \end{align*}
                    where the above inequality holds for any $N \in \mathbb{N}$, by Lemma \ref{weightineq}. Setting $N = 2d$ completes the proof.
                \end{proof}
                The next lemma will be crucially used in the proof of Proposition \ref{mainproppointwise}.
            
            \begin{lemma}\label{poissonsummationsub}
                Let $\{x_j\}_{i \in I}$ be a $R^{-\epsilon}$- separated subset of $B_{R}$. Then
                \begin{equation*}
                    \sum_{j \in I} \abs{Eg(x_j)}^2 \lesssim_{\epsilon} R^{\epsilon d}\int_{B_{2R}}\abs{Eg(x)}^2 dx + R^{-10000d}\norm{g}_{L^2(\Sigma)}^2
                \end{equation*}
                \end{lemma}
                \begin{proof}
                    By Lemma \ref{weightineq}, we have that $\abs{Eg(x_j)}^2 \lesssim \int_{\R^d}\abs{Eg(y)}^2w_{x_j}(y)dy$ for each $i \in I$. Thus
                    \begin{align*}
                        \sum_{j \in I} \abs{Eg(x_j)}^2 &\lesssim \int_{R^d} \abs{Eg(y)}^2 \left(\sum_{j \in I}w_j(y) \right) dy\\
                        &=\int_{B_{2R}} \abs{Eg(y)}^2 \left(\sum_{j \in I}w_j(y) \right) dy + \int_{\R^d \setminus B_{2R}} \abs{Eg(y)}^2 \left(\sum_{j \in I}w_j(y) \right) dy
                    \end{align*}
                    The pointwise weight inequality given in Lemma \ref{ptwise} shows that the first term is $\lesssim_{\epsilon} R^{\epsilon d}\int_{B_{2R}}\abs{Eg(x)}^2 dx$. To estimate the second term, we note that if $\abs{y} \geq 2R$, then $\sum_{i \in I} w_{x_i}(y) \leq C_N \mathrm{card}(I)(\abs{y}-R)^{-N} \lesssim_{d} C_{N}R^{(1+\epsilon)d} (\abs{y}-R)^{N}$. Using the H\"older and Tomas-Stein inequalities we have:
                    \begin{align*}
                        \int_{\R^d \setminus B_{2R}} \abs{Eg(y)}^2 \left(\sum_{i \in I}w_i(y) \right) dy &\lesssim \norm{Eg}_{\frac{2d+2}{d-1}}^2 C_N\left(\int_{\abs{y} \geq 2R}(\abs{y}-R)^{-\frac{N(d+1)}{2}}\right)^{\frac{2}{d+1}}\\
                        &\lesssim_{\epsilon, d, N} \norm{g}_{L^2(\Sigma)}^2\left(\int_{\abs{y} \geq R} u^{-\frac{N(d+1)}{2}}du\right)^{\frac{2}{d+1}}
                    \end{align*}
                    Choosing $N$ to be sufficiently large (e.g. $N=10^{6}d$) will result in the integral above being bounded by (up to an absolute constant depending on $d$) $R^{-10000d}$.
                \end{proof}
                \begin{remark}
                    One can replace $Eg$ with $Eg_{i_1,i_2,...,i_l}$ in the above proof, and an analogous statement is true.
                \end{remark}
        At many stages in the proof we rely on this classical fact that the Mizohata--Takeuchi Conjecture is known to be true when $w$ is the characteristic function of $B_R$. This is referred to in the literature as the Agmon--H\"ormander Trace inequality. The proof is well known and is essentially a consequence of Plancherel's theorem, but we record it here for completeness.
        \begin{proposition}[Agmon--H\"ormander Trace Inequality, M--T conjecture is true for $w(x)= \mathbbm{1}_{B_R}$]
            \begin{equation*}
                \int_{B_R}\abs{Eg(x)}^2dx \lesssim R\int_{\Sigma}\abs{g(\omega)}^2d\sigma(\omega).
            \end{equation*}
        \end{proposition}
            Since $\Sigma$ is a compact hypersurface, we may (locally) parameterize $\Sigma$ as $(\omega',\Sigma(\omega'))$ where $\omega' = (\omega_1,...,\omega_{d-1})$ with $\omega' \in [0,1]^{d-1}$ (say). Denote the spatial variable $x$ as $(x',x_d)$, where $x'=(x_1,x_2,...,x_{d-1})$. We may write
            \begin{equation*}
                \int_{B_R} \abs{Eg(x)}^2dx = \int_{B_R} \abs{\int_{[0,1]^{d-1}}g(w',\Sigma(\omega')) e^{2\pi i \omega' \cdot x'} e^{2\pi i \Sigma(\omega') \cdot x_d }N(\omega') d \omega'}^2 dx'dx_d
            \end{equation*}
        where $\abs{N(\omega')} \lesssim_{d} 1$ for all $\omega' \in [0,1]^{d-1}$, with $\norm{g}_{L^2(\Sigma)} \sim \norm{\tilde{g}}_{L^2[0,1]^{d-1}}$, where $\tilde{g}(\omega') \coloneqq g(\omega, \Sigma(\omega))$. By integrating in $x'$ first and using Plancherel's theorem (with repsect to the function $g(w',\Sigma(\omega'))e^{2 \pi i \Sigma(\omega') \cdot x_d}N(\omega')$), and then integrating in $x_d$, we see that the above expression can by bounded by $\sim R\norm{g}_{{L^2(\Sigma)}}^2$ as desired.
        \begin{remark}
            A similar statement is true for the operator $Eg_{i_1,i_2,...,i_l}$ as well, by the exact same proof.
        \end{remark}
\newpage
\bibliographystyle{plain}
\bibliography{references}
\end{document}